\documentclass[leqno,10pt]{amsart}
\usepackage{latexsym,amssymb,amsmath,graphics,color}
\usepackage[dvips]{graphicx}

\makeatletter
\@namedef{subjclassname@2010}{%
  \textup{2010} Mathematics Subject Classification}
\makeatother

\newtheorem{thm}{Theorem}[section]
\newtheorem{cor}[thm]{Corollary}
\newtheorem{lem}[thm]{Lemma}
\newtheorem{prop}[thm]{Proposition}

\theoremstyle{definition}
\newtheorem{defin}[thm]{Definition}

\newtheorem{rem}[thm]{Remark}
\newtheorem{exa}[thm]{Example}
\newtheorem*{notation}{Notation}

\numberwithin{equation}{section}
\newcommand{\8}{\infty}
\newcommand{\eps}{\varepsilon}
\newcommand{\s}{\sigma}

\usepackage{color}
\newcommand{\red}{\color{black}}
\newcommand{\blue}{\color{black}}

\renewcommand{\d}{\mathrm{d}}

\newcommand{\E}{\mathbb{E}}
\newcommand{\Prob}{\mathbb{P}}

\newcommand{\Pfs}[1][]{\ensuremath{\mathbb{P}_{#1}\text{-a.s.}}}
\newcommand{\Erw}[2][]{\ensuremath{\mathbb{E}_{#1} \left( {#2} \right)}}

\newcommand{\N}{\mathbb{N}}
\newcommand{\Z}{\mathbb{Z}}
\newcommand{\Q}{\mathbb{Q}}
\newcommand{\R}{\mathbb{R}}
\newcommand{\B}{\mathfrak{B}}
\newcommand{\C}{\mathbb{C}}

\newcommand{\F}{\mathcal{F}}
\renewcommand{\S}{\mathcal{S}}

\newcommand{\llam}{{l}}

\newcommand{\supp}{\mathrm{supp}\,  }

\renewcommand{\epsilon}{\varepsilon}
\renewcommand{\rho}{\varrho}

\newcommand{\1}[1][]{\mathbf{1}_{#1}}
\newcommand{\norm}[1]{\ensuremath{\left\| {#1} \right\|}}

\newcommand{\abs}[1]{\ensuremath{\left| {#1} \right|}}
\newcommand{\skalar}[1]{\langle #1 \rangle}

\renewcommand{\k}[1][0]{\ensuremath{ {\kappa_{#1}}}}

\newcommand{\eqdist}{\stackrel{d}{=}}

\newcommand{\Case}[1][]{\textsc{Case #1}}
\newcommand{\Step}[1][]{\textsc{Step #1}}

\newcommand{\Pset}[1][]{\mathcal{P}_{#1}}

\newcommand{\Rd}{\R^d}
\newcommand{\Rdnn}{{\R^d_\ge}}
\newcommand{\Rdp}{{\R^d_>}}

\newcommand{\Cf}[2][]{\mathcal{C}^{#1}\left( #2 \right)}
\newcommand{\Cbf}[2][]{\mathcal{C}_b^{#1}\left( #2 \right)}

\renewcommand{\B}{\mathcal{B}}

\newcommand{\condC}{(C)}
\newcommand{\QQ}{\mathbb{Q}}
\newcommand{\Sp}{\mathbb{S}_\ge}
\newcommand{\Sd}{\mathbb{S}}
\newcommand{\Mset}{\mathcal{M}_+}
\newcommand{\interior}[1]{\mathrm{int}({#1})}
\newcommand{\est}[1][s]{r^*_{#1}}
\newcommand{\Pst}[1][s]{P_*^{#1}}
\newcommand{\nust}[1][s]{\nu^*_{#1}}

\newcommand{\pis}[1][s]{\pi^{#1}}
\newcommand{\es}[1][s]{r_{#1}}
\newcommand{\Ps}[1][s]{P^{#1}}

\newcommand{\Qs}[1][s]{Q^{#1}}

\newcommand{\nus}[1][s]{\nu_{#1}}
\newcommand{\mM}{\mathbf{M}}

\newcommand{\mPi}{\matrix{\Pi}}

\newcommand{\mA}{\matrix{A}}
\renewcommand{\matrix}[1]{\mathbf{#1}}

\newcommand{\mb}{\matrix{b}}

\newcommand{\as}{\cdot}

\newcommand{\QP}[2][]{\mathbb{Q}_{#1}\left( #2 \right)}

\newcommand{\ma}{\matrix{a}}
\newcommand{\closure}[1]{\overline{#1}}

\renewcommand{\P}[2][]{\ensuremath{\mathbb{P}_{#1} \left( {#2} \right)}}

\newcommand{\norma}[1]{\left[\left[ #1 \right]\right]}
\newcommand{\normb}[1]{\left[ #1 \right]}
\newcommand{\pb}{{\mathbb P}}

\newcommand{\Qxs}{{\QQ_x^s}}
\newcommand{\QQP}[2][]{\Q_{#1} \left( #2 \right)}

\newcommand{\Qys}{{\QQ_y^s}}
\newcommand{\mcL}{\mathcal{L}}

\newcommand{\ip}{\textit{(i-p)}}
\newcommand{\ipo}{\textit{(i-p)}}
\newcommand{\ide}{\textit{(id)}}
\newcommand{\Pd}{\mathbb{P}^{d-1}}

\newcommand{\wt}{\widetilde}

\renewcommand{\F}{{\mathcal F}}

\newcommand{\mG}{{\blue\matrix{\daleth}}}
\newcommand{\mm}{\matrix{m}}

\begin{document}





\title{Precise Large Deviation Results for Products of Random Matrices}

    \author{Dariusz Buraczewski, Sebastian Mentemeier}
 \address{ Uniwersytet Wroc\l awski \\ Instytut Matematyczny \\ pl. Grunwaldzki 2/4 \\ 50-384 Wroc\l aw, Poland}
 \email{\{dbura,mente\}@math.uni.wroc.pl}


\begin{abstract}
The theorem of Furstenberg and Kesten provides a strong law of large numbers for the norm of a product of random matrices. This can be extended under various assumptions, covering  nonnegative as well as invertible matrices, to a law of large numbers for the norm of a vector on which the matrices act. We prove corresponding precise large deviation results, generalizing the Bahadur-Rao theorem to this situation. Therefore, we obtain a third-order Edgeworth expansion for the cumulative distribution function of the vector norm. This result in turn relies on an application of the Nagaev-Guivarch method.
Our result is then used to study matrix recursions, arising e.g. in financial time series, and to provide precise large deviation estimates there.
\end{abstract}

\subjclass[2010]{Primary 60F10; secondary 60H25}

\keywords{Products of random matrices, limit theorems, large deviations, random difference equations, Edgeworth expansion, Fourier techniques, Markov chains with general state space, Markov random walks, heavy tails}

\thanks{
D.~Buraczewski was partially supported by the NCN
grant DEC-2012/05/B/ST1/00692.
}
\maketitle

\section{Introduction}
Let $d \ge 1$, $\abs{\cdot}$ be any norm on $\R^d$ and $\norm{\cdot}$ be the corresponding operator norm. Let $(\mA_n)_{n \in \N}$ be a sequence of independent identically distributed  $d \times d$-matrices such that $\E \log^+ \norm{\mA_1} < \infty$. The Furstenberg-Kesten theorem {\red \cite[Theorem 2]{Furstenberg1960}} provides us with a strong law of large numbers for the norm of the products $\mPi_n:= \mA_n \cdots \mA_1$, namely
$$ \lim_{n \to \infty}\frac{1}{n} \log \norm{ \mPi_n } = \gamma \qquad \Pfs,$$
with $\gamma = \inf_{m \in \N} m^{-1} \E\, {\red \log} \norm{\mPi_m}$ being called the (top) Lyapunov exponent of $(\mA_n)_{n \in \N}$.
Under different sets of additional assumptions (to be detailed below) on the law $\mu$ of $A_1$, the convergence result has been strengthened towards a SLLN for the norm of a vector under the action of the random matrices: For example, following  \cite{Cohn1993,Hennion1997,Kesten1973,Kingman1973}, assume that the support of $\mu$ consists of nonnegative matrices and contains a matrix with all entries positive.  Then it holds for all nonnegative vectors $x$ that
\begin{equation}\label{furstenbergkesten} \lim_{n \to \infty}\frac{1}{n} S_n^x := \lim_{n \to \infty}\frac{1}{n} \log  \abs{\mPi_n x}   = \gamma \qquad \Pfs  \end{equation}
Under a second moment assumption, Hennion \cite{Hennion1997} proved a CLT, namely that
$$ \frac{1}{\sqrt{n}} \left(S_n^x - n \gamma\right)$$ converges to a normal law. For related limit theorems for invertible matrices, see \cite{Bougerol1985,LePage1982}.

Observe that in both cases, the SLLN and the CLT, the limit does not depend on the starting vector $x$. In contrast therewith is the result of Kesten \cite{Kesten1973} about the behavior of the maximum of $S_n^x$: Assuming in essence that the action of $\mA_1$ is both expanding and contracting with positive probability, that is  $\gamma < 0$ but $\P{S_1^x > 0}>0$, Kesten showed that there is $\alpha > 0$ and a continuous function $\es[]$ on the unit sphere $\Sd$, which is strictly positive on nonnegative vectors, such that
\begin{equation}\label{eq:Kesten max} \lim_{t \to \infty} e^{\alpha t}\, \P{\max_{n} S_n^x > t } = \es[](x).
\end{equation}
Here the behavior in the limit depends on the initial value.

We are going to provide a third-order Edgeworth expansion, which gives a rate of convergence for the CLT. We also provide a formula for the asymptotic variance $\sigma^2$ and show that it is positive under a natural nonlattice assumption. The Edgeworth expansion will as well be the main tool in describing the convergence in the law of large numbers, i.e. the Furstenberg-Kesten theorem, in more details.  In particular, we will discover how fluctuations depend on the starting vector $x$ as well as on the action of $(\mA_n)_{n \in \N}$ on the unit sphere, which is given by the Markov chain
$$ X_n^x := \frac{\mPi_n x}{\abs{\mPi_n x}}.$$
What we will prove is a large deviation result similar to the Bahadur-Rao theorem, i.e. for (suitable) $q > \gamma$, there is an explicitly given sequence $J_n(q)$ tending to infinity at an exponential rate, such that
\begin{equation} \label{eq:adhocBahadurRao} \lim_{n \to \infty} J_n(q) \Erw{r_q(X_n^x)\1[{\{{S_n^x} \ge nq\}}]} = r_q(x) \end{equation}
for a positive continuous function $r_q$ which depends on $q$, and generalizes the function $\es[](x)$ of Kesten's result.
This result is in the scope of large deviation principles for Markov additive processes, see \cite{Iltis2000,Kontoyiannis2003,Ney1987} for related results, where stronger conditions on $X_n^x$ have to be imposed than those who are satisfied for the chain generated by matrices. The very recent paper of Guivarc'h \cite{Guivarch2014} provides a local limit theorem, which is proved along similar lines as our Edgeworth expansion.


As an application of our result, we will  shed new light on the classical result of Kesten about random difference equations: Let {\red$\mM$} be a random $d \times d$-matrix and $B$ a random vector in $\R^d$.  Under  weak assumptions on {\red$(\mM,B)$}, there is a unique solution (in law) to the equation
\begin{equation}\label{eq:rde1} R \eqdist {\red\mM} R + B,\end{equation}
where $\eqdist$ means same law.
In the case of nonnegative ${\red \mM}, B, R$, Kesten \cite{Kesten1973} proved, assuming that ${\red \mM}$ is both contracting and expanding with positive probability, that for the same $\alpha > 0$ and $\es[]$ as in \eqref{eq:Kesten max},
\begin{equation}\label{eq:tails} \lim_{t \to \infty} \, t^\alpha \, \P{\skalar{R,x} > t} = K \es[](x), \end{equation}
for some $K >0$. This result has been extended to the case of invertible matrices in \cite{AM2010,BDGHU2009, Guivarch2012,Klueppelberg2004,LePage1983}, where it has always been an involved question to prove that $K$ is actually positive.
In both cases (nonnegative resp. invertible matrices), our result will be applied to   give an rather elementary proof of the fact that $K > 0$. {\red Here, the law of the matrix $\mA := \mM^\top$ will be relevant.}

This approach can also be extended to the study of branching equations, i.e.
\begin{equation}\label{eq:branching} R \eqdist \sum_{i=1}^N {\red \mM_i} R_i + B, \end{equation}
where now $N \ge 2$ is a fixed integer, {\red $(\mM_1, \dots, \mM_N)$} are random matrices and $B$ a random vector, independent of $R_i$, which are i.i.d. copies of $R$.  For random variables $R$ satisfying such an equation, the heavy tail property \eqref{eq:tails} has been shown to hold in \cite{BDGM2014,BDMM2013,Mirek2013}, but the positivity of $K$ remained a partially open question in the latter two articles. Due to the branching structure of Eq. \eqref{eq:branching}, the combinatorial part of the approach becomes more involved (it has been worked out in the one-dimensional case in \cite{BDZ2014+}), this is why we decided to postpone it to the separate work \cite{BM2015} and focus on the application of the large deviations result here, which can be seen more directly in the case of Eq. \eqref{eq:rde1}.

%


Having thus described the scope of the paper, we are now going to introduce some notations and concepts in order to state the main results in full detail. Since we want to solve questions concerned with nonnegative matrices as well as with invertible matrices, we are led to introduce several sets of assumptions (namely those of Kesten \cite{Kesten1973}, Guivarc'h and Le Page \cite{Guivarch2012} and Alsmeyer and Mentemeier \cite{AM2010}) on the law $\mu$ of the random matrix $\mA$, with all of them being sufficient for the announced results to hold. The main focus will be on nonnegative matrices, for which we will provide details of proofs, while for invertible matrices, we will mainly highlight the differences and refer to the works cited above.

\section{Notations and Preliminaries}

We start by introducing three sets of assumptions for random matrices. Let $d \ge 1$. Given a probability law $\mu$ on the set of $d \times d$-matrices $M(d \times d, \R)$, let $(\mA_n)_{n \in \N}$ be a sequence of i.i.d.~random matrices with law $\mu$ and write $\mPi_n:= \mA_n \cdots \mA_1$ for the $n$-fold product. Equip $\Rd$ with any norm $\abs{\cdot}$, write $\norm{\ma}:=\sup_{x \in \Sd} \abs{\ma x}$ for the operator norm of a matrix $\ma$  and denote the unit sphere in $\Rd$ by $\Sd$. We write
$$ \ma \as x := \frac{\ma x}{\abs{\ma x}}, \qquad x \in \Sd$$ for the action of a matrix $\ma$ on $\Sd$ (as soon as this is well defined). If $\Sd$ is invariant under the action of $\mA_1$, we introduce a Markov chain on $\Sd$ by
$$ X_n^x := \mPi_n \as x, \qquad x \in \Sd.$$

\subsection{Nonnegative Matrices: Condition $\condC$}

 Denote the cone of vectors with nonnegative entries by $\Rdnn$
 and write
$$ \Sp = \{ x \in \Rdnn \, : \, \abs{x}=1\}$$
for its intersection with unit sphere.
The set of $d \times d$-matrices with nonnegative entries is denoted by $\Mset$ and we write
$$\interior{\Mset} = \{ \ma \in M(d \times d, \R) \, : \, \ma_{i,j}>0 \ \forall\, 1 \le i,j \le d \}$$ for its interior, which consists of matrices that have all entries   positive.
A  matrix $\ma \in \Mset$ is called \emph{allowable} (see \cite{Hennion1997}), if every row and every column has a  positive entry.

If $\ma$ is an allowable matrix, then its action on $\Sp$ is well defined, and
moreover, the quantity
$$ \iota(\ma) := \min_{x \in \Sp} \abs{\ma x} > 0.$$

Consider now a probability distribution $\mu$ on $\Mset$. Write $[\supp \mu]$ for the subsemigroup generated by its support. We say that $\mu$ satisfies condition $\condC$, if:
\begin{enumerate}
\item Every $\ma \in [\supp \mu]$ is allowable.
\item $[\supp \mu] \cap \interior{\Mset} \neq \emptyset$.
\end{enumerate}
In the following, $\Gamma:= [\supp \mu]$. Observe that condition $\condC$ holds for $\Gamma$ if and only if it holds for $\Gamma^\top$.
Refering to the Perron-Frobenius theorem, every $\ma \in \interior{\Mset}$ possesses a unique dominant eigenvalue $\lambda_\ma$, (i.e.~$\abs{\lambda_\ma} > \abs{\lambda_i}$ for any other eigenvalue $\lambda_i$ of $\ma$) which is positive and algebraically simple, and a corresponding eigenvector $v_\ma \in \interior{\Sp}$. For a subsemigroup $\Gamma$ of allowable matrices, we define the collection of all such (normalized) dominant eigenvectors by
$$ V(\Gamma) := \closure{\left\{v_\ma \, : \, \ma \in \Gamma \cap \interior{\Mset}\right\}}. $$

It can be shown (see \cite[Lemma 4.3]{BDGM2014}) that $V(\Gamma)$ is the unique minimal $\Gamma$-invariant subset of $\Sp$, i.e.~every {\red closed} $\Gamma$-invariant subset of $\Sp$ contains $V(\Gamma)$. It is worth mentioning already now, that the Markov chain $X_n^x$ possesses a unique stationary probability measure, the support of which is given by $V(\Gamma)$.

\subsection{Invertible Matrices: Condition \ip}
In order to highlight connections, we decided to use the same symbols for objects which play the same role in the context of invertible matrices as they did for nonnegative matrices. The condition \ip (irreducible and proximal), described below, is due to Guivarc'h, Le Page and Raugi and was studied in detail in several articles by these authors, the most comprehensive one of which is \cite{Guivarch2012}.

Let now $\mu$ be a probability measure on the group $GL(d,\R)$ of invertible $d \times d$ matrices and $\Gamma$ be the closed semigroup of $GL(d,\R)$ generated by $\supp \, \mu$. A matrix $\ma$ with an algebraic simple dominant $\lambda_\ma$ is called \emph{proximal}. This replaces the notion of a matrix with strictly positive entries, which is always proximal by the Perron-Frobenius theorem. Then the measure $\mu$ is said to satisfy condition \ip, if
\begin{enumerate}
\item There is no finite union $\mathcal{W}=\bigcup_{i=1}^n W_i$ of subspaces $0 \neq W_i \subsetneq \R^d$ which is $\Gamma$-invariant, i.e. $\Gamma \mathcal{W}=\mathcal{W}$. (\emph{irreducibility}) 
\item $\Gamma$ contains a proximal matrix. (\emph{proximality})
\end{enumerate}
We will consider invertible matrices acting on the projective space $\Pd$ which is obtained from $\Sd$ by identifying $x$ with $-x$, i.e.
$$ \Pd~\simeq~ \Sd/\pm .$$ Studying the action of the matrices on $\Pd$ rather than on $\Sd$ has several technical advantages, for example, the definition
$$ V(\Gamma) := \closure{\left\{v_\ma \in \Pd \, : \, \ma \in \Gamma \text{ is proximal }\right\}}, $$
becomes unambiguous. Note that the norm $\abs{\ma x}$ for $x \in \Pd$ is well defined, since it does not depend on the choice of a representant of $x$ in $\Sd$.

For the case of invertible matrices, we have that
$$ \iota(\ma)~:=~ \inf_{x \in \Pd} \abs{\ma x} ~=~ \norm{\ma^{-1}}^{-1}.$$
%

\subsection{Invertible Matrices: Condition \ide}
The third set of assumptions, called \ide~for irreducible and density, appears first at the end of Kesten's work \cite{Kesten1973} and was elaborated  by Alsmeyer and Mentemeier in \cite{AM2010}. In fact, it can be shown to imply condition \ipo. Due to the stronger assumption that $\mu$ is absolutely continuous, it often allows for simpler proofs, this is why we include it as an extra set of assumptions.

Let $\mu$ be a probability measure on $GL(d,\R)$ and $(\mA_n)_{n \in \N}$ be an i.i.d.~sequence with law $\mu$ and write $\mPi_n:=\mA_n \cdots \mA_1$. Then $\mu$ is said to satisfy condition $\ide$ if
\begin{enumerate}
\item for all open $U \subset \Sd$ and all $x \in \Sd$, there is $n \in \N$ such that $\P{\mPi_n \as x \in U} >0$, and 
\item there are a matrix $\ma_0 \in GL(d,\R)$, $\delta, c >0$ and $n_0 \in \N$ such that
$$ 
 \P{\mPi_{n_0} \in d\ma } ~\ge~ c \1[{B_\delta(\ma_0)}](\ma) \, \llam(d \ma) ,$$
where $\llam$ denotes the Lebesgue measure on $\R^{d^2} \simeq M(d \times d,\R)$. 
\end{enumerate}
The classical example is $\mu$ having a density about the identity matrix.

It is shown in \cite[Lemma 5.5]{AM2010} that $X_n^x$ is a  Doeblin chain under condition \ide. The support of its  stationary probability measure is $\Sd$ by {\red \cite[Proposition 4.3]{BDMM2013}}, therefore we are led to identify $V(\Gamma):=\Sd$ in the case of \ide.

\subsection{Basic properties for all cases}\label{sect:defQ} Below, we identify $\S=\Sp$ in the case of nonnegative matrices, $\S=\Pd$ in the case of \ip-matrices and $\S=\Sd$ in the case of \ide-matrices. Given a measure $\mu$ on matrices, set
$$ I_\mu := \{ s \ge 0 \, : \, \int_{} \norm{\ma}^s \, \mu(\d \ma) <\infty \}.$$
Then, for $s \in I_\mu$, we define an operator  in the set $\Cf{\S}$ of continuous functions on $\S$ by
\begin{equation}\label{eq:Ps}  \Ps f(x) := \int_{} \abs{\ma x}^s \, f(\ma \as x) \, \mu(\d \ma), \end{equation}
and the 'transposed' operator by
\begin{equation}\label{eq:Pst}  \Pst f(x) := \int_{} \abs{\ma^\top x}^s \, f(\ma^\top \as x) \, \mu(\d \ma). \end{equation}
Properties of both operators, which will be given in a moment, will be important in our results. 
Beforehand, we introduce a function that will turn out to describe the spectral radius of these operators.

On $I_\mu$, define the log-convex function 
\begin{equation}\label{def:ks}
k(s) ~:=~ \lim_{n \to \infty} \left(\E{\norm{\mA_n \ldots \mA_1 }^s} \right)^\frac{1}{n} {\red ~=~  \lim_{n \to \infty} \left(\E{\norm{\mA_n^\top  \ldots \mA_1^\top }^s} \right)^\frac{1}{n}}.
\end{equation} {\red Here the second identity holds since  $\norm{\ma}=\norm{\ma^\top}$ and the $(\mA_i)_{i\in \N}$ are i.i.d. }
We have the following result:

\begin{prop}\label{prop:transferoperators}
Assume that $\mu$ satisfies $\condC$, \ipo~or \ide~ and let $s \in I_\mu$.
\begin{enumerate}
\item Then the spectral radii $\rho(\Ps)$ and $\rho(\Pst)$ both equal $k(s)$. \label{c1}
\item There is a unique normalized  function $\es \in \Cf{\S}$ and a unique probability measure $\nus \in \Pset(\S)$ satisfying
$$ \Ps \es = k(s) \es \quad \text{ and } \quad \Ps \nus = k(s) \nus.$$
\item The function $\es$ is strictly positive and $\bar{s}:=\min\{s,1\}$-H\"older continuous and $\supp\, \nus = V(\Gamma)$. \label{suppnus}
\item {\red If $\nust$ is a probability measure satisfying $\Pst \nust = k(s) \nust$, then there is $c>0$ such that \label{nust}
$$ \es(x) ~=~ c \int_\S \abs{\skalar{x,y}}^s \nust(dy).$$}
\item The function $s \mapsto k(s)$ is log-convex on $I_\mu$, hence continuous on $\interior{I_\mu}$ with left- and right derivatives. \label{c4}
\item The function $s \mapsto k(s)$ is analytic on $\interior{I_\mu}$. \label{c5}
\end{enumerate}
\end{prop}
{\red 
\begin{proof}[Source:]
Claims \eqref{c1}--\eqref{c4} were proved in \cite[Proposition 3.1]{BDGM2014} for nonnegative matrices, in \cite[Theorem 2.6 and Theorem 2.17]{Guivarch2012} for invertible matrices under condition \ipo~ and in \cite[Theorem 17.1]{Mentemeier2013a} under condition \ide. The analycity of $k(s)$ in assertion \eqref{c5} is proved by using perturbation theory, and was proved first under condition \ipo~ in \cite[Corollary 3.20]{Guivarch2012}, and subsequently, using the same methods, in \cite[Corollary 4.12]{BDMM2013} under condition \ide~ and is proved below in Corollary \ref{cor:k} for nonnegative matrices. 
\end{proof}

\begin{rem}
Given only finiteness of $$\E \bigg( \big( 1+\norm{\mA}^{s_0} \big) \big(\abs{\log \norm{\mA}} + \abs{\log \iota(\mA)} \big) \bigg),$$ the mapping $s \mapsto k(s)$ is still differentiable on the {\em closed} interval $[0,s_0]$, this has been proved in \cite[Theorem 3.10]{Guivarch2012} under condition \ipo~ and in \cite[Theorem 6.1]{BDGM2014} for nonnegative matrices. 
\end{rem}
}
Proposition \ref{prop:transferoperators} is crucial in order to define an exponential change of the measure $\mu$: Let $\Omega=M(d \times d, \R)^\N$ and $(\mA_n)_{n \in \N} : \Omega \to \Omega$ the fibered identity.
Now introducing for each $n$ the kernel
\begin{equation}
\label{qn}
q_n^s(x,\ma) = \frac{|\ma x|^s}{k^n(s)}\frac{\es(\ma\cdot x)}{\es(x)},
\end{equation}
we see that for each $x \in \S$ and $n \in \N$,
$$ \int q_n^s(x,\ma_n \cdots \ma_1) \mu^{\otimes n}(\d \ma_1, \dots, \d \ma_n)=1$$
{\red and the relation
\begin{equation}\label{eq:qns}
q_{n}^s(x, \ma) q_{m}^s(\ma \as x, \mb) = q_{n+m}^s(x, \mb \ma).
\end{equation}
}
Moreover, for each $x \in \S$ the sequence $q_n^s(x, \cdot) \mu^{\otimes n}$ of probability measures is projective, hence by the Kolmogorov extension theorem, it gives rise to a probability measure $\Qxs$ on $\Omega$, which we call the $s$-shifted measure. The corresponding expectation symbol is denoted by $\E_{\Qxs}$. Note that $(\mA_n)_{n \in \N}$ are i.i.d.~with law $\mu$ for $s=0$. {\red We use the symbol $\Q_x$ for $\Q_x^0$.}

With the conventions $\Qxs(\{X_0 = x\})=1$, we have the  Markov chain $X_n$ and the Markov additive process $S_n$:
\begin{eqnarray*}
 X_n & :=& \mA_n \as X_{n-1} = \frac{\mA_n X_{n-1}}{\abs{\mA_n X_{n-1}}},\\
  \quad S_n &:=& \log \abs{\mA_n \cdots \mA_1 X_0} {\red~=~ \log \abs{\mA_n X_{n-1}} + S_{n-1}  }.
\end{eqnarray*}
{\red The second identity shows that $(X_n, S_n)$ carries the structure of a {\em Markov Random Walk}, i.e. the law of the increments $S_n - S_{n-1}$ depends on the past only via $X_{n-1}$. }

Writing as before $\mPi_n := \mA_n \cdots \mA_1$, we have the following fundamental identities, valid for any bounded measurable function $f$ and $n \in \N$:
\begin{align}
\frac{1}{k(s)^n \es(x)}\Erw{f(x,\mA_1, \dots, \mA_n) \es(X_n^x) \abs{\mPi_n x}^s} = & \Erw[\Qxs]{f(X_0, \mA_1, \dots, \mA_n)}, \\
\frac{1}{k(s)^n \es(x)}\Erw{f\Bigl((X_k^x, S_k^x)_{k=0}^n \Bigr) \es(X_n^x) \abs{\mPi_n x}^s} = & \Erw[\Qxs]{f\Bigl((X_k, S_k)_{k=0}^n \Bigr)}. \label{eq:com1}
\end{align}

The transition operator of $(X_n)_{n \in \N}$ is given by
\begin{equation}\label{eq:defQs} \Qs f(x) := \frac{1}{\es(x) k(s)} \Ps (f \cdot \es)(x). \end{equation}
It follows from Proposition \ref{prop:transferoperators} that $\Qs$ has a unique stationary probability measure
$$\pi^s := \frac{\es\nu^s}{\nu^s(\es)}$$ with support $V(\Gamma)$.
We set $$\QQ^s := \int \Qxs \, \pis(dx).$$

\subsection{On $S_n^x$.}
Each of the assumptions introduced above is sufficient for the announced extension of the Furstenberg-Kesten theorem to hold:
\begin{prop}\label{prop:FK}
Assume that $\mu$ satisfies $\condC$ or \ip~ or \ide. Let $s \in  \{0\} \cup  \interior{I_\mu}$ and assume there is $0 < \epsilon < 1$ such that
\begin{equation}\label{eq:iotamoment}
\E \norm{\mA}^{s+\epsilon} \iota(\mA)^{-\epsilon} < \infty.
\end{equation}
Then it holds that $q:= \E_{\QQ^s} S_1 = k'(s)/k(s) \in \R$, and
$$ {\red \lim_{n \to \infty} \frac{1}{n} \log \norm{\mPi_n} ~=~} \lim_{n \to \infty} \frac{S_n}{n} ~=~ q \qquad \Qxs\text{-a.s.}$$
for all $x \in \S$.
\end{prop}
This is proved in \cite[Theorem 6.1]{BDGM2014} under condition $\condC$, in \cite[Theorem 3.10]{Guivarch2012} under condition \ipo~ and in \cite[Proposition 20.2]{Mentemeier2013a} under condition \ide. In the last reference, the first identity is not proved, but it follows from the corresponding result for \ipo.

\begin{rem}
Recall from Proposition \ref{prop:transferoperators} that $k(s)$ is log-convex. Therefore,
$$ \Lambda(s)~:=~ \log k(s)$$ is convex, and
$$ q ~=~ \frac{k'(s)}{k(s)} ~=~ \Lambda' (s) ~\ge~  \Lambda'(0) ~=~ \frac{k'(0)}{k(0)} ~=~\gamma. $$
The function
$$ \Lambda^*(q) ~:=~ sq - \Lambda(s)  $$
is the Fenchel-Legendre transform of $\Lambda$ and nondecreasing on $I_\mu$, see \cite[Lemma 2.2.5]{Dembo1998}. In particular, it is nonnegative on $I_\mu$.
\end{rem}

When studying random walks, an important distinction is between so-called lattice types, i.e. whether or not the random walk takes values only in some lattice $c \Z$ for $c \ge 0$ . A similar concept applies for Markov random walks, which are introduced below. The lattice type of $S_n$ only depends on the support of $\mu$, thus we give first a measure-free definition, which implies the more frequently used subsequent definition, which is relative to the measure {\red $\Q^s$}.

\begin{defin}\label{lem:non-arithmetic}
\begin{enumerate}
\item We say that $\Gamma$ resp. $\mu$ is \emph{arithmetic}, if there is $t>0$ together with $\theta \in [0, 2\pi)$ and a function $\vartheta : \Sp \to \R$ such that
\begin{equation}\label{eq:non-arithmetic}\tag{A}
\forall \ma \in \Gamma, \ \forall x \in V(\Gamma) \ : \ \exp\Bigl(i t \log \abs{\ma x} - i \theta + i (\vartheta(\ma \as x) - \vartheta(x)) \Bigr) =1.
\end{equation}
If no such $t$ exists, then $\Gamma$ is said to be \emph{non-arithmetic}.
\item The Markov random walk $(X_n, S_n)$ is said to be arithmetic under {\red $\Q^s$}, if there is $t>0$ together with $\theta \in [0, 2\pi)$ and a function $\vartheta : \S \to \R$ such that
\begin{equation}\label{eq:arithmetic2}
\E_{\Q^s}\exp\Bigl(i t S_1 - i \theta + i (\vartheta(X_1) - \vartheta(X_0)) \Bigr) =1,
\end{equation}
and non-arithmetic otherwise.
\end{enumerate}
\end{defin}

We have the following implications.

\begin{lem}\label{lem:implications_arithmetic}
If $\Gamma ~=~ [\supp \mu]$ is arithmetic, then $(X_n, S_n)$ is arithmetic under each {\red $\Q^s$} with the same $t, \theta, \vartheta$.
Conversely, if $(X_n, S_n)$ is arithmetic under some {\red $\Q^s$} and the function $\vartheta$ is continuous on $\S$, then $\Gamma$ is arithmetic as well with the same $t, \theta, \vartheta$. 
\end{lem}

\begin{proof}
Recalling that $\supp \pis = V(\Gamma)$, we observe that Eq. \eqref{eq:arithmetic2} is equivalent to
$$ \exp\Bigl(i t \log{\ma x} - i \theta + i (\vartheta(\ma \as x) - \vartheta(x)) \Bigr) ~=~1 \qquad \text{ for $\mu$-a.e. $\ma \in \supp\, \mu$ and $\pis$-a.e. $x \in V(\Gamma)$ },$$
i.e. for dense subsets of $\supp\, \mu$ resp. $V(\Gamma)$, which gives the asserted implications.
\end{proof}

It is shown in \cite[Proposition 4.6]{Guivarch+Urban:2005} that under condition \ip, $\Gamma=[\supp\, \mu]$ is non-arithmetic, while it is shown in \cite[Lemma 5.8]{AM2010}, that $(X_n, S_n)$ is {\red non-arithmetic} under each {\red $\Q^s$}  under condition $\ide$.

 A simple sufficient condition (due to Kesten \cite{Kesten1973}) for $\Gamma$ to be non-arithmetic under condition $\condC$ is the following.
Set $$ S(\Gamma) := \{\log \lambda_\ma \ : \ \ma \in \Gamma \cap \interior{\Mset} \}.$$

\begin{lem}
Assume that the (additive) subgroup of $\R$ generated by $S(\Gamma)$ is dense.
Then $\mu$ is non-arithmetic.
\end{lem}

\begin{proof}
Supposing that Eq. \eqref{eq:non-arithmetic} holds for some $t, \theta$ and $\vartheta$, then we have for any $\ma \in \Gamma \cap \interior{\Mset}$ that $v_\ma \in V(\Gamma)$, hence
$$ \exp\Bigl(i\bigl[ t \log \abs{\ma v_\ma} - \theta +(\vartheta(\ma \as v_\ma) - \vartheta(v_\ma))   \bigr] \Bigr] =
 e^{i(t\log \lambda_{\ma} - \theta)}
$$
Consequently, for any $\ma, \matrix{h} \in \Gamma \cap \interior{\Mset}$,
$$\log \lambda_\ma -\log \lambda_\matrix{h} \in \frac{2 \pi}{t} \Z.$$
But by our assumption, $S(\Gamma)$
is not contained in $\frac{2 \pi}{t} \Z$ for any $t >0$; this gives a contradiction.
\end{proof}

\begin{cor}
If there are $\ma, \mb \in \Gamma \cap \interior{\Mset}$ with $\frac{\log \lambda_\ma}{\log \lambda_\mb} \notin \Q$, then $\mu$ is non-arithmetic.
\end{cor}

Now we have enough notation to state our main results.

\section{Statement of main results}

We will prove the following analogue of the Bahadur-Rao theorem for products of random matrices. The role of the cumulant generating function is played here by $\Lambda(s) =\log k(s)$.

\begin{thm}\label{thm:BahadurRao}
Assume that $\mu$ satisfies $\condC$ and is non-arithmetic, or that $\mu$ satisfies \ipo~ or \ide.
If $q=\E_{\QQ^s} S_1 = \Lambda'(s) $ for some $s \in \interior{I_\mu}$ and there is $0 < \epsilon < 1$ such that
\eqref{eq:iotamoment} holds, then
$$ \lim_{n \to \infty}\, \sup_{x \in \S} \abs{ \sqrt{n}\, e^{n\Lambda^*(q)}\, J(s)\, \Erw{\es(X_n^x) \1[\{ S_n^x \ge n q \}]} - \es(x)} ~=~0, $$
where $$ J(s) ~=~ {s \sigma \sqrt{2 \pi} },  \qquad \text{ with }  \sigma^2 ~=~ \Lambda''(s) ~=~ \lim_{n \to \infty} \frac{1}{n} \E_{\QQ^s} (S_n-nq)^2 > 0.$$\end{thm}


Since the function $\es$ is strictly positive and continuous on the compact set $\S$, hence bounded, this gives in particular uniform bounds for the large deviation probabilities:

\begin{cor}
There are $0 < c \le C < \infty$ such that for all $x \in \S$,
$$ c ~\le~ \liminf_{n \to \infty}~ \sqrt{n}\,(e^{sq})^n\,\P{ S_n^x \ge n q } ~\le~ \limsup_{n \to \infty}~ \sqrt{n}\,(e^{sq})^n\,\P{ S_n^x \ge n q } ~\le~ C. $$
\end{cor}

These large deviations results will be used to prove the following result about random difference equations, which gives an elementary proof that the tail estimates derived e.g. in \cite{Kesten1973,Klueppelberg2004,AM2010,Guivarch2012} are precise:

\begin{thm}\label{thm:rde}
Let {\red $\mM$ be a random matrix and let $B$ be a random vector in $\Rd$. Write $\mA:=\mM^\top$ and denote by $\mu$ the law of $\mA$. Assume that $k'(0)<0$ and that there is $\alpha \in \interior{I_\mu}$ with $k(\alpha)=1$ and }
\begin{equation}\label{eq:moment conditions}
\E \norm{\mA}^{\alpha+\epsilon} \iota(\mA)^{-\epsilon} < \infty, \qquad 0 < \E \abs{B}^{\alpha + \epsilon} < \infty
\end{equation}
for some $\epsilon >0$.  There is a random variable $R$,  unique in distribution, satisfying {\red $R \eqdist \mM R + B$}.
\begin{enumerate}

\item \label{nonnegative} Let {\red $\mA$} be nonnegative, satisfying condition $\condC$ and being non-arithmetic. Assume that $\supp R \cap \Rdnn$ is unbounded.  Then there is $\delta > 0$ such that for all $x \in \Sp$,
$$ \liminf_{t \to \infty} \, t^\alpha \P{\skalar{x,R}>t} \ge \delta.$$
\item Let $\mA \in GL(d,\R)$, satisfying  \ide. Assume that $\P{\mA r + B = r} < 1$ for all $r \in \Rd$.
 Then there is $\delta > 0$ such that for all $x \in \Sd$,
$$ \liminf_{t \to \infty} \, t^\alpha \P{\skalar{x,R}>t} \ge \delta.$$
\item {\red Let  $\mA \in GL(d,\R)$, satisfying \ip. Assume that $\Gamma^*$ does not leave invariant any proper closed convex cone in $\Rd$, and that $\P{\mA r + B = r} < 1$ for all $r \in \Rd$.
 Then there is $\delta > 0$ such that for all $x \in \Sd$,
$$ \liminf_{t \to \infty} \, t^\alpha \P{\skalar{x,R}>t} \ge \delta.$$}
\end{enumerate}
\end{thm}

In this theorems, {\red we impose the assumptions on the law of $\mA=\mM^\top$ rather than on the law of $\mM$ (note nevertheless, that $\condC$ or \ip~ hold for $\mM^\top$ as soon as they hold for $\mM$). The reason is as follows: Let $(\mM_k, B_k)_{k \in \N}$ be a sequence of i.i.d.~copies of $(\mM,B)$. Then, upon iterating Eq. \eqref{eq:rde1}, we obtain  {\red $R \eqdist \mM_1 \cdots \mM_n R + \sum_{k\le n} \mM_1 \cdots \mM_{k-1}B_k$}, which leads to the study of
$$ \skalar{x,R} ~\eqdist~ \skalar{x,\mM_1 \cdots \mM_n R + \sum_{k\le n} \mM_1 \cdots \mM_{k-1}B_k} ~=~ \skalar{\mM_n^\top \cdots \mM_1^\top x, R} + \ldots,$$
and we are going to show that the first term dominates in order to use Theorem \ref{thm:BahadurRao} to derive estimates.}

\begin{rem}\label{rem:rde}
Let us stress that in \eqref{nonnegative} we do not assume that $B$ is nonnegative and that the condition  $\supp R \cap \Rdnn$ being unbounded is obviously also necessary for the heavy tail property. Thereby, we generalize the result of Kesten, namely \cite[Theorem 3]{Kesten1973}.
A sufficient condition for $\supp R \cap \Rdnn$ being unbounded is $B$ being nonnegative, or {\red $\mM,B$} being independent and $\P{B \in \Rdp} > 0$.
\end{rem}

\begin{rem}
The law of the random variable $R$ is given by
{\red $ \sum_{k=1}^\infty \mM_1 \cdots \mM_{k-1} B_k,$}
from which we immediately obtain the estimate (for $s \ge 1$)
$$ {\red (\E \abs{R}^s)^{1/s} ~ \le ~ \sum_{k=1}^\infty \left( \E \norm{\mM_1 \cdots \mM_n}^s \right)^{1/s} \, (\E \abs{B}^s)^{1/s} ~=~ \sum_{k=1}^\infty \left( \E \norm{\mM_1^\top \cdots \mM_n^\top}^s \right)^{1/s} \, (\E \abs{B}^s)^{1/s}.}$$
This shows that if {\red $k(s) <1$} and $\E \abs{B}^s < \infty$, then readily $\E \abs{R}^s < \infty$, which shows in particular that under the assumptions of Theorem \ref{thm:rde},
$$ \limsup_{t \to \infty} t^{s} \P{\skalar{x,R}>t} = 0  $$
for all $0 \le s < \alpha$ and all $x \in \mathcal{S}$.
\end{rem}

{\red
\begin{rem}
The moment conditions \eqref{eq:moment conditions} are not optimal, precise tail estimates have been obtained 
under the assumptions
$$ \E \norm{\mA}^\alpha ( \log \norm{\mA} + \abs{ \log \iota(\mA)}) < \infty, \qquad 0 < \E \abs{B}^\alpha < \infty, $$
see \cite[Remark after Theorem 5.2]{Guivarch2012} in the case of  \ipo~resp.~ \cite[Theorem 13.2]{Mentemeier2013a}  for the case of condition \ide.
\end{rem}
}

%
%

\subsection{Structure of the paper and sketch of proofs}

The proof of Theorem \ref{thm:BahadurRao} 
will rest upon a third-order Edgeworth expansion for the cdf
$$ F_{n,x}^s(t) := \Qxs\left\{ \frac{S_n - nq}{\sigma \sqrt{n}} \le t \right\},$$ which is given in Theorem \ref{thm:edgeworth}.

To prove this intermediate result, we will use the Nagaev-Guivarc'h spectral method as in Hennion and Herv\'e \cite{HH2001} and Herv\'e and Pen\`e \cite{Herve2010}: The classical  Edgeworth expansion for random walks 
can be proved using the Fourier transform of $S_n$, in particular its behavior at zero.
{\red Upon introducing (for suitable $z \in \C$) the operator $Q(z)$ in $\Cf{\S}$ by
$$Q(z) f(x) ~:=~ \frac{1}{\es(x) k(s)} \int_{\Gamma} \abs{\ma x}^{s +z} f(\ma \as x) \es(\ma \as x) \, \mu(\d \ma) 
$$
we have the following fundamental identity for the Fourier transform $\phi_{n,x}$ of $\Qxs\{S_n \in \cdot \}$:
\begin{equation} \phi_{n,x}(t):= \E_{\Q_x^s}\left( e^{itS_n}\right) = \E_{\Q_x^s}\left( e^{itS_n} \, \1[\S](X_n) \right) = Q(it)^n \1[\S](x). \label{eq:FT} \end{equation}
This identity is a consequence of the following lemma.
\begin{lem}\label{lem:FT}
Let $\mA$ be a random matrix with law $\mu$, and assume that $\E \norm{\mA}^{s+\Re z} < \infty$ for $s >0, z \in \C$. Then the following identity holds for all $f \in \Cf{\S}$:
\begin{equation}  \E_{\Q_x^s}\left( e^{zS_n} f(X_n)\right) ~=~  Q(z)^n f(x) \label{eq:Qzt} \end{equation}
\end{lem}
\begin{proof} The assumption guarantees that $Q(z)$ is well defined, and all integrals appearing below are finite.
We use induction. For $n=1$, this is immediate from the definition of $Q(z)$ and identity \eqref{eq:com1}. Suppose \eqref{eq:FT} holds for $n \in \N$. Then, using again \eqref{eq:com1} and the fact that the $(\mA_i)$ are i.i.d~ with law $\mu$ under $\Prob$, we obtain
\begin{align*}
Q(z)^{n+1}f(x) ~&=~ Q(z) \left( Q(z)^n f)(x) \right) \\
&=~ \int_{\Gamma} \frac{\es(\ma \as x) \abs{\ma x}^{s+z}}{\es(x) k(s)} \E_{\Q_{\ma \as x}^s} (e^{z S_n} f(X_n)) \, \mu(\d \ma) \\
&=~ \int_{\Gamma} \frac{\es(\ma \as x) \abs{\ma x}^{s+z}}{\es(x) k(s)} \frac{1}{\es(\ma \as x) k(s)^n} \E \left( \abs{\mPi_n (\ma \as x)}^{s+z} \es( \mPi_n \as (\ma \as x))  f( \mPi_n \as (\ma \as x)) \right) \, \mu(\d \ma) \\
&=~ \frac{1}{\es(x) k(s)^{n+1}} \int_{\Gamma}   \E \left( \abs{\mPi_n \ma  x}^{s+z} \es( \mPi_n\ma \as x) f( \mPi_n\ma \as x) \right) \, \mu(\d \ma) \\
&=~  \frac{1}{\es(x) k(s)^{n+1}} \E \left( \abs{\mPi_{n+1} x}^{s+z} \es( \mPi_{n+1} \as x) .\, f( \mPi_{n+1} \as x) \right) \, \mu(\d \ma) \\
&=~ \E_{\Q_x^s} \left( e^{z S_{n+1}} f(X_{n+1})\right).
\end{align*}
\end{proof}

Observe that $Q(0)=Q^s$ and that, given $s \in \interior{I_\mu}$, the mapping $z \mapsto Q(z)$ is holomorphic in some domain. }We are going to show that the operator $\Qs$ is quasi-compact with a simple dominant eigenvalue $\theta(0)=1$, and thereupon, using holomorphic perturbation theory, the decomposition
$$ Q^n(z) = \theta(z)^n M(z) + L(z)^n,$$
for a rank-one projection $M$ and an operator $L(z)$ with spectral radius $\rho(L(z)) < \rho(Q(z))$.
From this we will finally deduce that for $n \to \infty$,
$$ \phi_{n,x}(t/\sqrt{n}) = Q^n(it/\sqrt{n}) \1[\S](x) \approx \theta(it\sqrt{n}),$$
i.e. behavior at zero of the Fourier transforms is given by small perturbations of the dominant eigenvalue of $\Qs$.

\medskip
Therefore, we start our investigations by proving spectral properties of $\Qs$ and the family $Q(z)$ (in the case of nonnegative matrices). In Section \ref{sect:quasi-compact}, we prove, continuing \cite{BDGM2014} and based on the approach in \cite{Guivarch2012}, that $Q^s$ is quasi-compact. This property is needed in order to apply a perturbation theorem which proves the decomposition of the family $Q(z)$ in Section \ref{sect:perturbation}. Then we are ready to prove a third-order Edgeworth expansion for $F_{n,x}^s$ in Section \ref{sect:edgeworth}, which is used to prove Theorem \ref{thm:BahadurRao} in Section \ref{sect:BahadurRao}. Sections \ref{sect:nonarithmetic} and \ref{sect:taylor} study the implications of the non-arithmeticity condition, as well as formulas for $\sigma^2$.

Section \ref{sect:rde} is concerned with Theorem \ref{thm:rde}. We start by providing an example, namely the ARCH(q)-process, to which our results apply and continue by giving an outline of the proof of Theorem \ref{thm:rde}, while we postpone the technical details to the final Section \ref{sect:rde proofs}.

\section{Quasi-compactness of $\Qs$}\label{sect:quasi-compact}

\subsection{Nonnegative matrices} In this section, which is based on the approach of Guivarc'h and Le Page \cite{Guivarch2012} for \ip, we are going to prove that for each  $s \in I_\mu$, the operator $\Qs$ is quasi-compact (has a spectral gap) on a subspace of $\Cf{\Sp}$, namely the space of functions that are $\bar{s}:=\min\{s,1\}$-H\"older continuous with respect to a particular metric $d$ on $\Sp$. At first, we will recall the Theorem of Ionescu Tulcea and Marinescu, which will be used in order to prove the quasi-compactness. Then we introduce the particular metric $d$ which will be useful when finally checking the assumptions of this theorem. 

\medskip

We write $\mcL(\B,\B)$ for the set of all bounded linear operators from $\B$ to $\B$. An operator $Q \in \mcL(\B,\B)$ is said to be {\em quasi-compact} if $\B$ can be decomposed into two {\red closed} $Q$-invariant subspaces $\B = E \oplus F$ where the spectral radius $\rho(Q_{|F}) < \rho(Q)$ while $\dim E < \infty$ and each eigenvalue of $Q_{|E}$ has modulus $\rho(Q)$.

Subsequently, a convenient way to prove the quasi-compactness of $\Qs$ will be to use the following generalization of the {\red Theorem of Ionescu-Tulcea and Marinescu}:

\begin{thm}[{\cite[Theorem II.5]{HH2001}}]\label{thm:ITM}
Let $(\B, \norma{\cdot})$ be a Banach space {\red and let $\normb{\cdot}$ be a continuous semi-norm on $\B$. Assume that $Q$ is a bounded} operator in $\B$ such that
\begin{enumerate}
\item $Q\,  \{f \, : \, \norma{f} \le 1 \}$ is conditionally compact in $(\B, \normb{\cdot})$. \label{prop1}
\item there exists a constant $M$ such that for all $f \in \B$, $\normb{Q f} \le M \normb{f}$, \label{prop2}
\item there exist $k \in \N$ and real numbers $r$ and $R$ with $r < \rho(Q)$ and, for all $f \in \B$,\label{prop3}
$$ \norma{Q^k f} \le R \normb{f} + r^k \norma{f}.$$
\end{enumerate}
Then $Q$ is quasi-compact.
\end{thm}

Though we have not yet defined the metric $d$ on $\Sp$, let us nevertheless state right now, which Banach space and what norms we are going to consider. For $f \in \Cf{\Sp}$, set
$$ \normb{f}:= \sup_{x \in \Sp} \abs{f(x)}, \qquad \abs{f}_s :=\sup_{x,y \in \Sp} \frac{\abs{f(x) - f(y)}}{d(x,y)^{\bar{s}}}, \qquad \norma{f}:=\normb{f}+\abs{f}_s.$$
We consider the Banach space
$$ \B := \{ f \in \Cf{\Sp} \, : \, \abs{f}_s < \infty \} = \{ f \in \Cf{\Sp} \, : \, \norma{f} < \infty \}$$
equipped with the norm $\norma{\cdot}$.
Using 
Theorem \ref{thm:ITM}, we are going to prove the following:

\begin{prop}\label{prop:Qs}
Assume that $\mu$ {\red satisfies} $\condC$ and let $s \in I_\mu$. Then $Q^s \in \mcL(\B,\B)$, and there is an operator $N \in \mcL(\B,\B)$ with spectral radius $\rho(N) < 1$, such that
\begin{equation} \label{eq:decompQs} (Q^s)^n = M + N^n \end{equation}
for all $n \in \N$, where $M$ is a rank-one projection onto $\R\1[\Sp]$ with $M(f)(x) = \pis(f)$ for all $f \in \B$ and $x \in \Sp$.
\end{prop}

This will be done by a series of Lemmata, which will make use of the particular metric $d$ on $\Sp$, which we are going to introduce next.

\subsubsection{A metric on $\Sp$}\label{subsect:metric}

Given $x \neq y \in \Sp$, consider the  line $L$ trough these points. {\red Then $L \cap \partial \Rdnn$ consists of two points which we label by $a$ and $b$ in such a way that if we write $x= u_1 a + u_2 b$ and $y=v_1 a + v_2 b$ $u_1,u_2,v_1,v_2 \ge 0$ as convex combinations of $a$ and $b$, then $u_1 > v_1$, i.e. $x$ lies between $a$ and $y$. Then  the cross-ratio of $a,b$ and $x,y$ is given as
$$ [a,b;x,y] = \frac{u_2 v_1}{u_1 v_2} .$$}
%
%
The formulae $$d(x,y):=\phi([a,b;x,y])$$ for $\phi(s):= \frac{1-s}{1+s}$, $s \in [0,1]$,
defines a bounded distance on the unit sphere.
Its properties are summarized in the following {\red Proposition}.

\begin{prop} \label{prop:properties of d}
For any norm $\abs{\cdot}$, $d$ is a metric on $\Sp$ with
\begin{itemize}
\item $\sup\{d(x,y) \, : \, x,y \in \Sp \} =1$,
\item There is $C>0$ s.t. $d(x,y) \ge C \abs{x-y}$.
\end{itemize}
For $\ma \in \Mset$, there exists $c(\ma) \le 1$ such that:
\begin{enumerate}
\item $d(\ma \as x, \ma \as y) \le c(\ma) d(x,y),$
\item $c(\ma) <1$ if and only if $\ma \in \interior{\Mset}$,
\item if $\ma' \in \Mset$, then $c(\ma \ma') \le c(\ma) c(\ma')$,
\item $c(\ma^\top)=c(\ma)$.
\end{enumerate}
\end{prop}

{\red \begin{proof}[Source:] This is \cite[Proposition 3.1]{Hennion1997}. 
There, the results are stated relative to the 1-norm $\norm{x}_1 = \sum_{i=1}^n \abs{x_i}$ on $\Rd$, but they do in fact hold for any norm on $\Rd$, the main reason being that the cross-ratio is an projective invariant and thus independent of the shape of the unit sphere, and that all norms on $\Rd$ are comparable.
\end{proof}
}
The crucial properties of the metric $d$ are (1) and (2), saying that the action of nonnegative (positive) matrices is a (strict) contraction with respect to $d$.

\subsubsection{Checking the assumptions of the Ionescu-Tulcea-Marinescu theorem}

Let us first recall the definition of $\Qs$ in \eqref{eq:defQs}, from which we obtain the following formula for its iterates:
\begin{equation}
\label{eq:Qsn} (\Qs)^n f(x) ~=~ \Erw[\Qxs]{f(X_n)} ~=~ \Erw{q_n^s(x,\mPi_n) f(\mPi_n \as x)}.
\end{equation}

In order to prove Assumption (1) of Theorem \ref{thm:ITM}, we are going to apply the Arzel\`a-Ascoli theorem. Therefore, we have to prove equicontinuity of the family $\{ \Qs f \, : \, \norma{f} < \8 \}.$
This will follow from the subsequent estimates for the kernels $q_n^s$, {\red where it is shown} in particular, that the mappings $q_n^s(\cdot, \ma)$ are $\bar{s}$-H\"older on $\Sp$ for any {\red $\ma \in \Mset$}.

\begin{lem}\label{lem:prop:qn}
{\red Under the assumptions of Proposition \ref{prop:Qs}, there is} $C_s < \infty$ such that for all $n \in \N$, $x,y \in \Sp$, $\ma \in \Mset$
$$ \abs{q_n^s(x,\ma) - q_n^s(y,\ma)} \le C_s \frac{\norm{\ma}^s}{k(s)^n} d(x,y)^{\bar{s}}.$$
On the other hand, there is $c_s$ such that for all allowable $\ma$,
{\red $$ q_n^s(\ma) := \int q_n^s(x, \ma) \pis(\d x) \ge \frac{c_s}{k(s)^n} \norm{\ma}^s$$}
\end{lem}

\begin{proof}
Observe that by Proposition \ref{prop:properties of d}, any function that is H\"older-continuous on $(\Sp, \abs{\cdot})$ is as well H\"older-continuous on  $(\Sp, d)$.
Using that thus $\es$ is $\bar{s}$-H\"older with constant $d_{\es}$ and bounded with $0<d_1 \le \es(x) \le d_2<\infty$ for all $x \in \Sp$,  as well as property (2) of Proposition \ref{prop:properties of d}, we estimate
\begin{align*}
& \abs{\frac{\es(\ma \as x)}{\es(x)} \frac{\abs{\ma x}^s}{k(s)^n} - \frac{\es(\ma \as y)}{\es(y)} \frac{\abs{\ma y}^s}{k(s)^n}} \\
\le~&   \abs{\frac1{\es(x)} - \frac{1}{\es(y)}} \frac{\es(\ma \as x) \abs{\ma x}^s}{k(s)^n}  +   \abs{\abs{\ma x}^s - \abs{\ma y}^s} \frac{\es(\ma \as x) }{\es(y) k(s)^n}
 + \abs{\es(\ma \as x) - \es(\ma \as y)}  \frac{\abs{\ma y}^s}{k(s)^n \es(y)} \\
 \le~&  \frac{1}{d_1^2} \abs{\es(x)-\es(y)} \frac{d_2 \norm{\ma}^s}{k(s)^n} + \abs{\abs{\ma x}^s  -\abs{\ma y}^s} \frac{d_2}{d_1 k(s)^n} + d_{\es} \abs{x-y}^{\bar{s}}  \frac{\norm{\ma}^s}{k(s)^n d_1} \\
 \le~& \left(\frac{C d_{\es} d_2}{d_1^2} + \frac{C d_{\es}}{d_1} \right) \frac{\norm{\ma}^s}{k(s)^n} d(x,y)^{\bar{s}} + \frac{d_2}{d_1 k(s)^n)} \abs{\abs{\ma x}^s - {\red \abs{\ma y}^s}}.
\end{align*}
The last term has to be estimated differently for $s \le 1$ and $s >1$. If $s \le 1$, then
$$ \abs{\abs{\ma x}^s - {\red \abs{\ma y}^s}} \le \abs{\abs{\ma x} - \abs{\ma y}}^s \le \norm{\ma}^s \abs{x -y}^s.$$
If $s >1$, then
$$ \abs{\abs{\ma x}^s - {\red \abs{\ma y}^s}} \le \abs{\abs{\ma x} - \abs{\ma y}} \cdot s \cdot \max\{\abs{\ma x}^{s-1}, \abs{\ma y}^{s-1} \}  \le s\norm{\ma} \abs{x -y}^{\bar{s}} \norm{\ma}^{s-1}  $$

For the second part, recall $K = \inf\{ \es(x) / \es(y) \, : \, x,y \in \Sp \} >0$,
hence
\begin{align*}
q_n^s(\ma) \ge \frac{K}{k(s)^n} \int \abs{\ma x}^s \pis(\d x).
\end{align*}
It suffices to prove that $g(\ma):= \int \abs{\ma x}^s \pis(\d x) \ge c_s$ for all nonnegative $\ma$ with $\norm{\ma}=1$. On the compact set $\norm{\ma}=1$, $g$ attains its infimum. But if there is $\ma_0$ with $\int \abs{\ma_0 x}^s \pis(\d x) = 0$, then
$$ V(\Gamma) \subset \supp \nus = \supp \pis \subset \mathrm{ker}(\ma_0). $$
But since $\ma_0$ is a nonzero nonnegative matrix, $\mathrm{ker}(\ma_0) \cap \interior{\Sp} = \emptyset$, which gives a contradiction.
\end{proof}

{\red Let us note the following, surprising Corollary to Lemma \ref{lem:prop:qn}, which shows that the convergence in \eqref{def:ks} is exponentially fast.
\begin{cor}\label{cor:ksEs}
Under the assumptions of Proposition \ref{prop:Qs}, there is $c_s >0$ (the same as in Lemma \ref{lem:prop:qn}), such that
$$ k(s)^n \le \E \norm{\mPi_n}^s \le \frac{1}{c_s} k(s)^n \qquad \text{ for all $n \in \N$}.$$
\end{cor}

\begin{proof}
The first inequality holds since $k(s) = \lim_{n \to \infty} \big(\E \norm{\mPi_n}^s \big)^{1/n} = \inf_{m \in \N}  \big(\E \norm{\mPi_m}^s \big)^{1/m}$ due to submultiplicativity of the norm (see \cite[Theorem 1]{Furstenberg1960} for details). The second inequality holds by Lemma \ref{lem:prop:qn}, since $\E q_n^s(x,\mPi_n) =1$ for all $x \in \Sp$.
\end{proof}
}

Now we are ready to prove the following estimate, from which the validity of assumptions (1) and (3) will follow.

\begin{lem}\label{lem:hoeldercontinuous}
{\red Under the assumptions of Proposition \ref{prop:Qs}, there is} $C >0$ and a sequence $D(n)$ with $\lim_{n \to \infty} D(n) = 0$, such that for all $n \in \N$ and $f \in \B$,
\begin{align}\label{eq:hoeldercontinuous}
\abs{(Q^s)^n f}_s \le C \normb{f} + D(n) \abs{f}_s
\end{align}
\end{lem}

\begin{proof}
For all $f \in \B$, $\abs{f}_s < \infty$. For such $f$, we compute
\begin{eqnarray*}
\abs{(\Qs)^n f(x) - (\Qs)^n f(y)}
 &= & \abs{ \E_{\Qxs} f(X_n) - \E_{\Qys} f(X_n)} \\
&\le& \E_\Qxs\abs{ f(\mPi_n \as x) - f(\mPi_n \as y)  } + \abs{(\E_\Qxs - \E_\Qys) f(\mPi_n \as y)}\\ &=& I + II.
\end{eqnarray*}
Considering $I$,
\begin{align*}
I \le \abs{f}_s\, \E_\Qxs  d(\mPi_n \as x, \mPi_n \as y)^{\bar{s}} \le \abs{f}_s d(x,y)^{\bar{s}} \ \E_\Qxs c(\mPi_n)^{\bar{s}}.
\end{align*}
But due to Proposition \ref{prop:properties of d}, $c(\ma) \le 1$ for all $a \in \Mset$, and $c(\ma) < 1$ for $\ma \in \interior{\Mset}$. {\red By $\condC$, we have that
$\P{ \liminf_{n \to \infty } \left\{ \mPi_n \in \interior{\Mset} \right\}} = 1$ (see \cite[Lemma 3.1]{Hennion1997}), thus $c(\mPi_n) \to 0$ $\Qxs$-a.s.~ by Proposition \ref{prop:properties of d}, (2) and (3) and the boundedness of $c$. Moreover, $c$ is continuous on $\Mset$ by \cite[Lemma 10.8]{Hennion1997}. Therefore, we can use the dominated convergence theorem to infer
$$ \lim_{n \to \infty} D(n) := \lim_{n \to \infty}  \E_\Qxs c(\mPi_n)^{\bar{s}} = 0.$$}


Turning to $II$, we have, using Lemma \ref{lem:prop:qn},
$$ II \le [f] \E \big| q_n^s(x,\mPi_n)  - q_n^s(y,\mPi_n) \big|
\le \frac{[f] C_s d(x,y)^{\bar{s}}}{k(s)^n} \E \norm{\mPi_n}^s
\le [f] \frac{C_s}{c_s} d(x,y)^{\bar{s}} \E_{\QQ^s} \1
$$
Combining these estimates, we arrive at
\begin{align*}
\abs{(Q^s)^n f}_s \le \frac{C_s}{c_s} \normb{f} + \abs{f}_s
D(n).
\end{align*}
\end{proof}

\begin{proof}[Proof of Proposition \ref{prop:Qs}]
Now we are ready to show that Theorem \ref{thm:ITM} applies for $Q=\Qs$ with $\B$ as defined above.

\Step[1]:
Assumption \eqref{prop2} is satisfied for $M=1$, since $\Qs$ is a Markov operator on $(\Cf{\Sp},\normb{\cdot})$.

\medskip

\Step[2]: Assumption \eqref{prop1} holds for $\Qs$, i.e. $\Qs\{ f \, : \, \norma{f} \le 1 \}$ is conditionally compact in $(\B, \normb{\cdot})$. This is shown as follows. Since $\Qs$ is a Markov operator, and $\norma{f} \ge \normb{f}$, we have that $K:=\Qs\{ f \, : \, \norma{f} \le 1 \} \subset \{ f \in \B \, : \, \normb{f} \le 1 \}$, thus $K$ is bounded. Using \eqref{eq:hoeldercontinuous} with $n=1$, we deduce that the family $K$ is equicontinuous. Hence, applying the Arzel\`a-Ascoli theorem, $K$ is conditionally compact in $\Cf{\Sp}$ with respect to the topology of uniform convergence, i.e. w.r.t. $\normb{\cdot}$.

\medskip
\Step[3]: Next we show that Assumption \eqref{prop3} holds for $\Qs$, i.e. there exist $k \in \N$ and real numbers $r$ and $R$ with $r < \rho(\Qs)$ and, for all $f \in \B$,
$$ \norma{(\Qs)^k f} \le R \normb{f} + r^k \norma{f}.$$ In particular, $\Qs \in \mcL(\B,\B)$.

Observe, that it suffices to provide the estimate for \emph{one} $k \in \N$, it is not necessary to prove a geometric decay rate. Since the spectral radius $r(\Qs)=1$, it is enough to show that the inequality holds for some $r' < 1$ in the place of $r^k$, because then $r := {r'}^{1/k}<1$ satisfies the assumption.
Using \eqref{eq:hoeldercontinuous} and the fact that $\Qs$ is a Markov operator, we deduce that for any $n \in \N$,
\begin{align}
\label{eq:normaQs} \norma{(\Qs)^n f} ~&=~ \normb{{(\Qs)^n f}} + \abs{{(\Qs)^n f}}_s \\
~&\le~ \normb{f} + C \normb{f} + D(n) \abs{f}_s ~\le~ (1+C) \normb{f} + D(n) (\abs{f}_s + \normb{f}) \nonumber \\
~&=~ (1+C) \normb{f} + D(n) \norma{f} \nonumber .
\end{align}
But $D(n)$ tends to 0, thus we may choose $k$ such that {\red $D(k) < 1$}, and consequently, Assumption \eqref{prop3} is satisfied with $R:=1 + C$ and $r := D(k)^{1/k} < 1$.
 Thus {\red Theorem \ref{thm:ITM}} applies and gives the {\red quasi-compactness of $Q^s$ , i.e. $\B=E \oplus F$ for $Q^s$-invariant closed subspaces $E$ and $F$ with $\dim E <\infty$ and such that $Q^s_{|F}$ has spectral radius strictly smaller than 1, while each eigenvalue of $Q^s_{|E}$ has modulus 1. }

\medskip

\Step[4]: {\red Next we prove that $1$ is a simple eigenvalue}, and the only one of modulus one, i.e. $\dim E=1$.
It is shown in \cite[Theorem 4.13]{BDGM2014}, that for every $f \in \Cf{\Sp}$,
\begin{equation}\label{eq:convQs} \lim_{n \to \infty} (\Qs)^n f = \pis(f). \end{equation}
If now {\red$\Qs f = \lambda f$} with $\abs{\lambda}=1$,  then necessarily {\red $\lim_{n \to \infty} \lambda^{n} f \equiv \pis(\lambda f)$}, which implies $\lambda=1$ and $f={\rm const}$.

\medskip

{\red \Step[5]: We infer from Eq. \eqref{eq:normaQs} that
$$ \limsup_{n \to \infty} \rho(Q^s)^{-n} \Big( \sup \, \{\norma{(\Qs)^n f} \, : \, \norma{f} =1 \}\Big)  \le 1+C.$$
Therefore, $Q^s$ is {\em quasi compact of diagonal type} in the sense of \cite[Prop. III.1]{HH2001}. Consequently, \cite[Lemma III.3(v)]{HH2001} applies and gives the decomposition
\eqref{eq:decompQs} with $M$ being the projection on $E=\R\1[\Sp]$ with $M(f)=\pis(f)\1[\Sp]$ for all $f \in \B$, and $N:= Q^s -M$.
}
\end{proof}

\subsection{Invertible matrices}
As said before, the ideas of the proofs above were developed by Guivarc'h and Le Page for condition \ip, the result corresponding to Proposition \ref{prop:Qs} is \cite[Corollary 3.19]{Guivarch2012}. There, the distance $d(x,y)=\abs{x-y}$ for $x,y \in \Pd$ is the minimal euclidean distance between representants in $\Sd$.

Under assumption \ide, a corresponding decomposition, proved in  \cite[Proposition 4.3 and Lemma 4.11]{BDMM2013} holds on the (larger) space $\Cf{\Sd}$, for $(X_n)$ is a Doeblin chain  under each $\Qxs$.

\section{Non-artihmeticity and its consequences}\label{sect:nonarithmetic}

Subsequently, fix $s \in I_\mu$. We will now study implications of the non-arithmeticity and moment assumptions \eqref{eq:non-arithmetic} resp. \eqref{eq:iotamoment} for the family $Q(it)$. {\red Recall from Lemma \ref{lem:FT} the identity} 
\begin{equation}
\label{eq:defQit} Q(it)^n f(x) {\red ~=~} \Erw[\Qxs]{e^{it S_n} f(X_n)} ~=~ \Erw{q_n^s(x, \mPi_n) e^{it \log \abs{\mPi_n x}} f(\mPi_n \as x)}.
\end{equation}
This section is valid for all types of matrices. Recall that conditions \ipo~and \ide~readily imply non-arithmeticity.

Define $\B_\epsilon := \{ f \in \Cf{\S} \, : \, \abs{f}_\epsilon < \infty \}.$
Then we are going to prove the following result:

\begin{thm}\label{thm:prop:nonarithmetic}
Assume that $(X_n, S_n)$ is non-arithmetic under {\red $\Q^s$} or that $\mu$ is non-arithmetic. Assume that \eqref{eq:iotamoment} hold for some $\epsilon >0$. {\red Then $Q(it) \in \mcL(\B_\epsilon, \B_\epsilon)$ for all $t \in \R$. Moreover,  for all $t \neq 0$, the spectral radius $\rho(Q(it)) < 1$ and thus $1-Q(it)$ is invertible in $\mcL(\B_\epsilon, \B_\epsilon)$.}
\end{thm}

Considering the spectral radius, we have for all $n \in \N$ and all $f \in \Cf{\S}$ that
\begin{equation} \label{eq:Qitf}\abs{Q(it)^n f(x)} \le \E_{\Qxs}\abs{e^{itS_n} f(X_n)} = \E_{\Qxs} \abs{f(X_n)} = (Q^s)^n \abs{f}(x).\end{equation} Since $\B_\epsilon \subset \Cf{\S}$, this readily shows that $\rho(Q(it)) \le \rho(Q^s)=1$ for all $t \in \R$. To arrive at $\rho(Q(it)) < 1$, the main burden of the proof will be indeed to show that ${\red Q(it)} \in \mcL(\B_\epsilon,\B_\epsilon)$. This will be done by proving an estimate similar to \eqref{eq:hoeldercontinuous}. The first step in that direction is provided by the following lemma.

\begin{lem}\label{lem:epsilonHoelder}
Let $\ma$ either be an allowable nonnegative matrix or an invertible matrix. Then for all $t \in \R$, $x,y \in \S$ and $0 < \epsilon < 1$, the following estimate holds true:
\begin{equation}
\abs{e^{it \log \abs{\ma x}}- e^{it \log \abs{\ma y}}} \le D \abs{t}^{\epsilon} d\left( x,  y \right)^{\epsilon} \left( \frac{\norm{\ma}}{\iota(\ma)} \right)^{\epsilon}
\end{equation}
for some $D > 0$.
\end{lem}

Recall that in the case of nonnegative matrices, the distance $d$ on $\S$ was defined in \ref{subsect:metric}, while $d(x,y)$ equals the minimum of the euclidean distance of representants  of $x,y$ from $\Sd$ in the case of nonnegative matrices.

\begin{proof}
We start by noting the some useful inequalities:
\begin{align}\label{ineq1}
\frac12 \abs{e^{it}-e^{is}} = \frac{1}2 \abs{1-e^{i(s-t)}}\le \min\{1, \abs{t-s}\} \le \abs{t-s}^\beta
\end{align}
for all $t,s \in \R$, $\beta \in [0,1]$. Next, for all $a,b > 0$,
\begin{align}\label{ineq2}
\abs{\log a - \log b} = \abs{\int_a^b \, \frac{1}s \, \d s} \le \max\{\frac{1}{a}, \frac1b\}  \abs{a-b}.
\end{align}
Finally, if $\ma$ is allowable or invertible, then for all $x \in \S$, $\frac{1}{\abs{\ma x}} \le \frac{1}{\iota(\ma)}$. Putting these inequalities together, we conclude, using Proposition \ref{prop:properties of d} as well in the case of invertible matrices,
\begin{align*}
\abs{e^{it \log \abs{\ma x}}- e^{it \log \abs{\ma y}}}  ~&\le~ 2 \abs{t}^{\epsilon} \abs{\log \abs{\ma x} - \log \abs{\ma y}}^{\epsilon} \\
~&\le~ 2 \abs{t}^{\epsilon} \max\{\frac1{\abs{\ma x}}, \frac1{\abs{\ma y}} \}^{\epsilon} \ \abs{\abs{\ma x} - \abs{\ma y}}^{\epsilon} ~\le~ 2 \abs{t}^{\epsilon}\max\{\frac1{\abs{\ma x}}, \frac1{\abs{\ma y}} \}^{\epsilon} \ \norm{\ma}^{\epsilon} \abs{x-y}^{\epsilon}  \\
~&\le~ 2 \abs{t}^{\epsilon} C^{-1} d\left( x,  y \right)^{\epsilon} \left( \sup_{z \in \S} \frac{\norm{\ma}}{\abs{\ma z}} \right)^{\epsilon} ~\le~  2 \abs{t}^{\epsilon} C^{-1} d\left( x,  y \right)^{\epsilon} \left( \frac{\norm{\ma}}{\iota(\ma)} \right)^{\epsilon}
\end{align*}
\end{proof}

Now we are going to prove an estimate similar to \eqref{eq:hoeldercontinuous}:

\begin{lem}\label{lem:hoelder2}
There is $C >0$ and a sequence $D(n)$ with $\lim_{n \to \infty} D(n) = 0$, such that for all $n \in \N$ and $f \in \B$,
\begin{align}\label{eq:hoeldercontinuous2}
\abs{Q(it)^n f}_\epsilon \le C \normb{f} + D(n) \abs{f}_\epsilon
\end{align}
\end{lem}

\begin{proof}
 {\red\begin{eqnarray*}
 & & \abs{Q(it)^n f(x) - Q(it)^n f(y)} \\
 &= & \abs{ \Erw[\Qxs]{e^{it S_n} f(X_n)} - \Erw[\Qys]{e^{itS_n} f(X_n)}} \\
 &\le& \E \big( q_n^s(x,\mPi_n) \abs{ f(\mPi_n \as x) - f(\mPi_n \as y)  } \big) + \abs{\E \big(q_n^s(x,\mPi_n) - q_n^s(y,\mPi_n) \big) e^{it S_n^x} f(\mPi_n \as y)} \\
 & & + \normb{f} \E \big( q_n^s(y,\mPi_n) \abs{e^{it S_n^x} - e^{it S_n^y}} \big) \\ &=& I + II + III. 
\end{eqnarray*}}
Similar to the proof of Lemma \ref{lem:hoeldercontinuous}, we obtain the bounds
\begin{align*}
I ~&\le~ \abs{f}_\epsilon d(x,y)^\epsilon \E_{\Qxs} c(\mPi_n)^\epsilon =: D(n) \abs{f}_\epsilon d(x,y)^\epsilon, \\
II ~&\le~ \normb{f} C \, d(x,y)^\epsilon
\end{align*}
with $\lim_{n \to \infty} D(n) =0$.

Using Lemma \ref{lem:epsilonHoelder}, we deduce
\begin{align*}
III ~&\le~ D \normb{f} \abs{t}^{\epsilon} d\left( x,  y \right)^{\epsilon} \E_{\Qys}\left( \frac{\norm{\mPi_n}}{\iota(\mPi_n)} \right)^{\epsilon} \\
~&\le~ D' \normb{f} \abs{t}^{\epsilon} d\left( x,  y \right)^{\epsilon} \Erw{  \norm{\mPi_n}^{s+\epsilon} {\iota(\mPi_n)}^{-\epsilon}} \\
~&\le~ D' \normb{f} \abs{t}^{\epsilon} d\left( x,  y \right)^{\epsilon} \left(\E{  \norm{\mA_1}^{s+\epsilon} {\iota(\mA_1)}^{-\epsilon}}\right)^n,
\end{align*}
where the last expression is finite due to assumption \eqref{eq:iotamoment}.
\end{proof}

\medskip

%

\begin{proof}[Proof of Theorem \ref{thm:prop:nonarithmetic}]
Lemma \ref{lem:hoelder2} together with \eqref{eq:Qitf} proves that $Q(it)$ is a self-map of $\B_\epsilon$. Since {\red $\rho(Q(it)) \le 1$}, it remains to exclude the possibility $\rho(Q(it))=1$, which we will do by contradiction.

Assuming that the spectral radius $\rho(Q(it))=1$, one can proceed as in Section \ref{sect:quasi-compact} in order to show that the Ionescu-Tulcea-Marinescu theorem applies for $Q=Q(it)$ with
$$ \normb{f}:= \sup_{x \in \S} \abs{f(x)}, \qquad \norma{f}_\epsilon:=\normb{f}+\abs{f}_\epsilon$$
and the Banach space
$$ \B_\epsilon := \{ f \in \Cf{\S} \, : \, \abs{f}_\epsilon < \infty \} = \{ f \in \Cf{\S} \, : \, \norma{f}_\epsilon < \infty \}$$
equipped with the norm $\norma{\cdot}_\epsilon$.
The theorem yields that there has to be an eigenvalue with modulus equal to the spectral radius of $Q(it)$, i.e. with modulus equal to 1.

Hence, suppose there is an eigenfunction $f$ such that $Q(it) f = \lambda f$ with $\abs{\lambda}=1$. {\red Let $x_0 \in \Sp$ be such that $\abs{f(x_0)}=[f]$. Then \eqref{eq:Qitf} implies that $\abs{f}(x_0) \le ((Q^s)^n \abs{f})(x_0)$ and hence by \eqref{eq:convQs}, $\abs{f}(x_0) \le \pis(\abs{f})$. But the right hand side is a convex combination of $(\abs{f}(x))_{x \in V(\Gamma)}$ (see Proposition \ref{prop:transferoperators}, \eqref{suppnus}). Consequently, $\abs{f}$ has to be constant on $V(\Gamma)$.}
Thus, we can assume that $f(x)=e^{i \vartheta(x)}$ on $V(\Gamma)$ for a continuous function $\vartheta: \S \to \R$. Consequently,
$$ \E_{\Qxs} e^{i t S_1 + i \vartheta(X_1)} = e^{i \theta + i \vartheta(x)}.$$
But this contradicts the non-arithmeticity of $(X_n, S_n)$ under {\red $\Q^s$}; and, since $\vartheta$ is continuous, as well the nonarithmeticity of $\mu$, using Lemma \ref{lem:implications_arithmetic}.
\end{proof}

%
%

\noindent {\bf Remark}. Observe, that we only did prove the estimate
$$ \normb{Q(it)^n f}_\epsilon \le R \norma{f} + D(n) \normb{f}_\epsilon, $$
with $D(n)$ tending to 0. From this, property \eqref{prop3} of the Ionescu-Tulcea-Marinescu theorem can be deduced only if $\rho(Q(it))=1$, since otherwise, we do not know whether $r := D(n)^{1/n} < \rho(Q(it))$ holds for some $n$, since we do not know the rate of convergence for $D(n) \to 0$. So we do not know yet whether $Q(it)$ is quasi-compact for $t \neq 0$. Nevertheless, for small $t$, quasi-compactness will follow from the perturbation theorem below.

\section{The Perturbation Theorem}\label{sect:perturbation}
This section as well is valid for nonnegative and invertible matrices.
Recall {\red from Lemma \ref{lem:FT} the fundamental identity
$$ \phi_{n,x}(t) ~=~ \E_{\Q_x^s}\left( e^{itS_n}\right) ~=~ Q(it)^n \1[\S](x).$$}
In this section, we are going to apply an holomorphic perturbation theorem for $\Qs$ in order to show that (for small t)
$$ Q(it)^n = \theta(it)^n M(it) + N^n(it),$$
where $M(it)$ is a rank one-projection, which commutes with $N^n(it)$, and $\rho(N(it)) < \rho(Q(it))$. Using this decomposition, we will be -- roughly speaking -- able to replace
$$\phi_{n,x}(t/\sqrt{n}) \approx \theta(t\sqrt{n})$$ for large $n$ in the proof of the Edgeworth expansion.

Fix the parameter $s \in \interior{I_\mu}$ as well as $\epsilon$ such that \eqref{eq:iotamoment} is satisfied. By what has been shown above, $\Qs \in \mcL(\B_\epsilon, \B_\epsilon)$ is quasi-compact with a simple dominant eigenvalue $1$. This, and the holomorphicity of the mapping $z \mapsto Q(z)$, shown below,  will be the main ingredients for the application of a perturbation theorem.

\subsection{Perturbation theory for $\Qs$}

\begin{lem}\label{lem:defQz}
Choose $\delta>0$ such that $(s-\delta, s+ \delta) \subset I_\mu$ and $s+\delta > \epsilon$. Then for all
$z\in  H_\delta:=\{z \in \C \, : \, \Re z \in (-\delta, \delta) \}$, the operator $Q(z)$ on $\B_\epsilon$, which is given by
$$Q(z) f(x) := \frac{1}{\es(x) k(s)} \int \abs{\ma x}^{s +z} f(\ma \as x) \es(\ma \as x) \, \mu(\d \ma) = \E_\Qxs \big[e^{z S_1} f(X_1)\big],$$
is well defined. The mapping $Q : H_\delta \to \mathcal{L}(\B_\eps,\B_\eps)$, $z \mapsto Q(z)$ is holomorphic.
\end{lem}

\begin{proof}
Recalling that $\es$ is bounded from below and above, it follows that
\begin{equation}\label{eq:boundperturbation} [Q(z) f] \le K \E \norm{\mA}^{s + \Re z} < \infty, \end{equation}
since $s + \Re z \in I_\mu$.
Together with Lemma \ref{lem:hoelder2}, this proves that $Q(z) \in \mcL(\B_\epsilon,\B_\epsilon)$.

Now we can show that $z \mapsto Q(z)$ is weakly holomorphic, i.e. for any $f \in \B_\eps$, $\nu \in \B_\eps'$ ({\red the dual space of $\B_\eps$}), {\red $z \mapsto \int Q(z) f \d \nu$} is holomorphic. This readily implies that $z \mapsto Q(z)$ is (strongly) holomorphic, see \cite[Theorem V.3.1]{Yosida1980}.
In order to show weak holomorphicity, consider arbitrary $f, \nu$ and a closed curve $\gamma \subset B_\delta(0) \subset \C$. Then
\begin{align*}
&~\int_\gamma \left( \int_{\S} (Q(z) f)(x) \, \nu (\d x) \right) \ \d z \\
=&~ \int_{\S} \left( \frac{1}{\es(x) k(s)} \int  f(\ma \as x) \es(\ma \as x) \ \left\{\int_{\gamma} e^{(s+z) \ln \abs{\ma x}} \ \d z   \right\} \  \mu(\d \ma)            \right) \ \nu(\d x)
=0,
\end{align*}
for the innermost function is holomorphic in $z$. The change of the order of integration is guaranteed by the estimate \eqref{eq:boundperturbation}.
\end{proof}


Now we can apply the following perturbation theorem \cite[Theorem III.8]{HH2001}. 

\begin{thm}\label{theorem:perturbation}
Let $G_0:=B_\delta(0) \subset \C$ and let $(Q(z))_{z \in G_0}$ be a collection of elements of $\mathcal{L}(\B_\eps,\B_\eps)$ such that
\begin{itemize}
\item[(H1)] $z \mapsto Q(z)$ is holomorphic on $G_0$,
\item[(H2)] $Q(0)$ has one dominating simple eigenvalue and $\rho(Q(0))=1$.
\end{itemize}
Then there exist $G_1:=B_{\delta_1}(0) \subset \C$, $G_1 \subset G_0$ and holomorphic  mappings
$$ \theta : \, G_1 \to \C, \quad r :\, G_1 \to \B_\eps, \quad \nu :\, G_1 \to \B_\eps', \quad N: \, G_1 \to \mathcal{L}(\B_\eps,\B_\eps) $$
such that for all $n \ge 1$, $z \in G_1$
$$ Q^n(z) 
= \theta(z)^n M(z) + L(z)^n, $$
with $Q(z) r(z) = \theta(z) r(z)$ and $\nu(z) Q(z) = \theta(z) \nu(z)$.
Moreover, {\red for each $l_0 \in \N$} there exist constants $\eta_1, \eta_2 >0$, $c \ge 0$ such that for all $z \in G_1$,
$$ \abs{\theta(z)} \ge 1 - \eta_1, \text{ and } \max\left\{ \norm{\frac{\d^l}{\d z^l} L(z)^n} \, : \, {\red l \le l_0}\right\} \le c (1 - \eta_1 - \eta_2)^n.$$
\end{thm}

%
%
%

\subsection{The operators $R(t)$}
For the Edgeworth expansion, we will consider a slightly different operator, namely such that $S_1$ becomes centered: Let $q:= \E_{\Q_s} S_1$ denote the stationary drift of $S_1$ under $\Q_s$, and define the family $(R(t))_{t \in\R}$ of operators by
\begin{equation}\label{def:operator R}
R(t) f(x) := e^{-itq} Q(it)f(x) = \Erw[{\Qxs}]{e^{it(S_1 -q)} f(X_1)}.
\end{equation}

Upon defining
\begin{equation}
\lambda(t) := e^{-itq} \theta(it), \qquad N(t):= e^{-itq} L(it), \qquad \Pi(t) 
= M(it),
\end{equation}
we obtain the following corollary {\red of Lemma \ref{lem:defQz} and Theorem \ref{theorem:perturbation}}.

\begin{cor}\label{cor:Rt}
{\red There is $\delta_1 >0$} such that for all $t \in G:= (- \delta_1, \delta_1)$,
$$ R(t)^n = \lambda(t)^n \Pi(t) + N(t)^n,$$
with $\Pi(t) N(t) = N(t) \Pi(t) =0$. {\red For each $l_0 \in \N$ there is $\eta=\eta(l_0) >0$ and $c=c(l_0) < \infty$} such that  \begin{equation} \label{eq:boundN} \max\left\{ \norm{\frac{\d^l}{\d z^l} N(t)^n} \, : \, {\red l \le l_0} \right\} \le c (1 - \eta)^n. \end{equation}
The mappings $\lambda: G \to \C$, $\Pi : G \to \mathcal{L}(\B_\eps,\B_\eps)$ and $N: G \to \mathcal{L}(\B_\eps,\B_\eps)$ are $\mathcal{C}^\infty$, the latter ones in the strong operator sense.
\end{cor}
{\red For all purposes below, we can choose $l_0=3$, and may therefore consider $\eta=\eta(3)$, $c=c(3)$ fixed.}

In order to prove the Edgeworth expansion, we will make as well use of the following result{\red, which is inspired by \cite[Lemma 3.19]{BDG2010}.}

\begin{lem}\label{missing estimate}
Let $K \subset \R \setminus \{0\}$ be compact. Then for each $f \in \B_\epsilon$, there is $\rho < 1$ such that for all $t \in K$
\begin{equation} [R(t)^n f] \le \rho^n [f]. \end{equation}
\end{lem}

\begin{proof}
Fix $f \in \B_\eps$. For each $n \in \N$, the mapping
$$t \mapsto [R(t)^n f]^{1/n} = \left( \sup_{x \in S} \abs{\Erw{q_n^s(x,\mPi_n) e^{it(S_n^x -q)} f(X_n^x)}} \right)^{1/n}$$
is continuous. Hence, $\rho_f(t):= \limsup_{n \to \infty} [R(t)^n f]^{1/n} $ is upper semicontinuous, thus it attains it maximum on the compact set $K$, in $t_0  \neq 0$, say.
But  $\rho_f(t_0) \le \rho(R(t_0)) = \rho(Q(t_0)) < 1$, hence the assertion follows.
\end{proof}

\section{Taylor expansion of $\lambda(t)$ and positivity of the asymptotic variance}\label{sect:taylor}

In this section, which is valid for all types of matrices, we are going to relate the first and second order coefficients of the Taylor expansion of $\lambda$ with the expectation (which equals zero in fact) resp. the asymptotic variance of $S_n-nq$ under $\Qxs$. Moreover, we are able to prove that the asymptotic variance is positive as soon as $\mu$ is non-arithmetic.

\begin{lem}
\label{lem: lambda}\label{lem:sigma1}
Assume that $\mu$ satisfies $\condC$, \ipo~or \ide, and that $s \in \interior{I_\mu}$.
Then there is $\sigma \ge 0$ and $m_3 \in \R$ such that 
$${\red
\lambda(t) = 1  - \frac{\sigma^2}2\, t^2  -i \frac{m_3}{6}\, t^3 + o(t^3),}
$$ and
\begin{equation}
\label{eq: sigma}
\sigma^2 = \lim_{n \to \infty} \frac{1}{n} \E_{\QQ^s} (S_n-nq)^2  \qquad
 m_3= \lim_{n \to \infty} \frac{1}{n} \E_{\QQ^s} (S_n-nq)^3.
\end{equation}
For each $x\in \S$, the value
\begin{equation}
\label{eq: bx}
 b(x) = \lim_{n \to \infty} \E_\Qxs (S_n-nq)
 \end{equation}
is well defined, and the mapping $b \in \B_\eps$.
It holds that
\begin{equation}\label{eq:bx}
b(x) = \Erw[\Qxs]{(S_1-q) {\red +} b(X_1) }
\end{equation}
Moreover, \begin{equation}
 {\red \sigma^2 = \E_{\QQ^s} \left[  \Bigl( (S_1 - q) + b(X_1) \Bigr)^2 - b(X_1)^2 \right],} \end{equation}
and
\begin{equation}
\sup_{n \in \N} \abs{n\sigma^2 - \E_{\QQ^s}(S_n - nq)^2} < \infty. \label{eq:sn2finite}
\end{equation}
\end{lem}

 To prove  Lemma \ref{lem: lambda} we reason as in \cite[Lemma 8.3 \& Lemma 8.4 ]{Herve2010}.

\begin{proof}[Proof of Lemma \ref{lem: lambda}]

{\sc Step 1.} First we prove that $\lambda'(0)=0$.

Differentiating the equation
$
R(t)\Pi(t)\1 = \lambda(t)\Pi(t)\1
$ in the operator sense and computing its value at 0, we obtain
\begin{equation}
\label{eq: 1}
{\red R'(0)\1 + R(0)\Pi'(0)\1 = \lambda'(0)\1 + \Pi'(0)\1.}
\end{equation}
Both sides of the above equation are bounded continuous functions, so computing their integral with respect to the measure $\pi$ we obtain
$$\pi(R'(0)\1) +\pi(\Pi'(0)\1) = \lambda'(0) + \pi(\Pi'(0)\1)
$$
we have
$$
{\red \lambda'(0) = \pi(R'(0)\1) = i\E_{\Q^s}[S_1-q]=0.}
$$
{\sc Step 2. } Now we justify, that the function $b(x)$ is well defined as a function in {\red $\B_\epsilon$}.

{\red Observe first that by Lemma \ref{lem:FT} with $z=0$ and \eqref{eq:qns}, we have 
\begin{align}
(Q^s)^{n} f(x) ~&=~ \E_{\Qxs}[f(X_n)] \nonumber  \\
&=~ \E\, q_n^s(x, \mPi_n) \int_{\Gamma} q_1^s(\mPi_n \as x, \ma) \big( \log \abs{\ma (\mPi_n \as x)} -q \big) \, \mu(d\ma) \nonumber \\
&=~ \E \,  \int_{\Gamma} q_{n+1}^s(x, \ma \mPi_n) \big( \log \abs{\ma \mPi_n  x)} - \log \abs{ \mPi_n x} -q \big) \, \mu(d\ma) \nonumber \\
&=~ \E_{\Q_x^s} \left( S_{n+1} - S_n -q \right)~=~ \E_{\Qxs}[S_{n+1}] - \E_{\Qxs}[S_n] -q. \label{eq: 2}
\end{align}
  
      }

Next by \eqref{eq: 1} for any $k$ we have (recall $R(0)=Q^s$)
\begin{equation}\label{eq:3}
i(Q^s)^k \E_{\Qxs} [S_1-q] + (Q^s)^{k+1} \Pi'(0)\1(x) = (Q^s)^k \Pi'(0)\1(x).
\end{equation}
Hence summing over $k=0,1,\ldots, n-1$ we obtain
$$
i\sum_{k=0}^{   n-1} (Q^s)^{k}\E_{\Qxs}[S_1-q] +
\sum_{k=0}^{n-1} (Q^s)^{k+1} \Pi'(0)\1(x) = \sum_{k=0}^{n-1} (Q^s)^k \Pi'(0)\1(x).
$$
Thus by \eqref{eq: 2}
$$i\E_{\Qxs} [S_n-nq]  +(Q^s)^n \Pi'(0)\1(x) = \Pi'(0)\1(x).
$$

Since {\red $\Pi'(0)\1\in \B_\epsilon \subset \B$}, the limit $(Q^s)^n \Pi'(0)\1(x)$ exists by Proposition \ref{prop:Qs} and is equal to $\pi(\Pi'(0)\1)$. Deriving the equation $\Pi^2(t)=\Pi(t)$ and computing the result at 0 we obtain $$\pi(\Pi'(0)\1)=0.$$ Thus, the limit
$\lim\E_{\Qxs}[S_n-nq]$ exists, equals
\begin{equation}\label{eq:defbx} {\red b(x)~:=~ \frac{1}{i}\, \Pi'(0)\1(x)},
\end{equation}  and thus $b$ is well defined and is an element of $\B_\eps$, since $\Pi'(0)$ maps $\1$ into $\B_\eps$.
The formula \eqref{eq:bx} follows from \eqref{eq:3} for $k=0$.

{\sc Step 3.} Using the above, we obtain the following Taylor expansions, valid for small $t$:
$$ \lambda(t)^n = 1 + n \lambda''(0) \frac{t^2}{2} + n \lambda^{(3)} \frac{t^3}{6} + o(t^3),$$
$$ \pis\left(  \Pi(t) \1 \right) = 1 + d_1 \frac{t^2}{2} + d_2 \frac{t^3}{6} + o(t^3);$$
as well as the classical expansion for the characteristic function, i.e.
$$ \E_{\QQ^s} e^{it(S_n - nq)} ~=~ 1 - \E_{\QQ^s}(S_n - nq)^2 \, \frac{t^2}{2} - i \, \E_{\QQ^s} (S_n - nq)^3 \, \frac{t^3}6 + o(t^3).$$
From the fundamental identity,
$$ \E_{\QQ^s} e^{it(S_n - nq)}  ~=~ \pis(R(t)^n \1) = \lambda(t)^n \pis(\Pi(t)\1) + \pis(N(t)^n \1), $$
using {\red the bounds \eqref{eq:boundN} (with $l_0=3$) for $N$ as well, }
we deduce that
$$ n \lambda''(0) + d_1 + O((1-\eta)^n) ={\red - \E_{\QQ^s}(S_n - nq)^2}  $$
and
$$ n \lambda^{(3)}(0) + d_2 + O((1-\eta)^n) = {\red -i\,\E_{\QQ^s}(S_n - nq)^3} .$$
Hence, the identification of $\sigma^2$ and $m_3$ as well as the boundedness assertion follow.

\Step[4]: Finally, we provide the formula for $\sigma^2$. Differentiating $R(t)\Pi(t)\1 = \lambda(t)\Pi(t)\1$ twice and integrating against $\pis$, using {\red $\pis R(0) = \pis$}, we obtain
$$ \pis(R''(0)\1) + 2\, \pis(R'(0) \Pi'(0)\1) = \lambda''(0),$$
hence recalling from above that {\red $i \,b(x) = \Pi'(0)\1(x)$},
$$ {\red - }\int \E_{\Qxs} (\log \norm{\mA_1 x}-q)^2 \, \pis(\d x)  \, {\red - } \,  2\, \int \E_{\Qxs} \left[ \log (\norm{\mA_1 x}-q) \, b(\mA_1 \as x) \right] \, \pis(\d x) ~=~ \lambda''(0),$$
i.e.
$$ \sigma^2 = \E_{\QQ^s} \left[ (S_1 - q)^2 + 2 (S_1-q) b(X_1) \right]$$
and the result follows by quadratic extension inside the expectation.
\end{proof}

Using the above formula for $\sigma^2$, one can show that non-arithmeticity readily implies that $\sigma^2 >0$.

\begin{lem}\label{cor:arithmetic}
Assume that $\sigma^2 =0$, then
$$ S_1 ~=~ q - b(X_1) + b(X_0) \qquad \QQ^s\text{-a.s.}, $$
in particular, $(X_n, S_n)$ is arithmetic under $\QQ^s$ and $\mu$ is arithmetic.
\end{lem}

\begin{proof}
If $\sigma^2=0$, then it follows from \eqref{eq:sn2finite}, that
\begin{align*} \int b(x)^2 \, \pis(\d x) =&~ \int \left[\lim_{n \to \8} \E_{\Qxs}(S_n -nq) \right]^2 \pis(\d x)\le \int \liminf_{n \to \8} \E_{\Qxs}(S_n -nq)^2  \pis(\d x) \\
\le&~ \liminf_{n \to \infty} \int \E_{\Qxs}(S_n -nq)^2  \pis(\d x) \le \sup_{n \in \N} \E_{\QQ^s} (S_n - nq)^2 < \infty.
\end{align*}
Then we may rewrite the formula from Lemma \ref{lem:sigma1} to read
\begin{align}\label{eq:5}
\sigma^2 ~=&~ \E_{\QQ^s}  \Bigl( (S_1 - q) + b(X_1) \Bigr)^2 - \E_{\QQ^s} b(X_1)^2 = \E_{\QQ^s}  \Bigl( (S_1 - q) + b(X_1) \Bigr)^2 - \E_{\QQ^s} b(X_0)^2.
\end{align}
Using \eqref{eq:bx}, we see that
$$ \Erw[{\QQ^s}]{\Bigl((S_1-q)+b(X_1)\Bigr)b(X_0)} = \int \, b(x) \Erw[{\Qxs}]{\Bigl((S_1-q)+b(X_1)\Bigr)} \, \pis(\d x) = \int \, b(x)^2 \pis(\d x), 
$$
which we use in \eqref{eq:5} to obtain (through binomial formula) that
$$ 0 = \sigma^2 = \E_{\QQ^s}  \Bigl( (S_1 - q) + b(X_1) - b(X_0) \Bigr)^2 = \int_{V(\Gamma)} \pis(\d x) \, \int_{\supp \, \mu} \mu(\d \ma) \, \Bigl( \log \abs{\ma x} - q + b(\ma \as x) - b(x) \Bigr)^2. $$
This gives the assertion; and the arithmeticity of $\mu$ follows, since the function $b$ is continuous (see Lemma \ref{lem:implications_arithmetic}).
\end{proof}

%

Finally, we note some expressions for derivatives of $k$.

\begin{cor}\label{cor:k}
The function $k(s)$ is $\mathcal{C}^\infty$ on $\interior{I_\mu}$, and
$$ \frac{k'(s)}{k(s)}=q=\E_{\Q^s} S_1, \qquad  \frac{k^{(2)}(s)}{k(s)} = q^2 +\sigma^2. $$
\end{cor}

\begin{proof}
Recalling the definition of $\Ps$, we see that for $\epsilon \in (- \delta_1, \delta_1)$,

\begin{align*} (\Ps[s+\epsilon])^n f(x) &~= \es(x) k(s)^n Q(\epsilon) \frac{f}{\es} (x) \\
&~= k(s)^n \theta(\epsilon)^n \es(x) r(\epsilon)(x) \int_{\Sp} f(y)/ \es(y) \, \nu(\epsilon)(dy) + \es(x) k(s)^n N(\epsilon) \frac{f}{\es}(x)
\end{align*}

By Proposition \ref{prop:transferoperators}, $\Ps[s+\epsilon]$ has a unique strictly positive eigenfunction, which is then given by $r_s(x)r(\epsilon)(x)$ and thus the corresponding eigenvalue
equals
\begin{equation}\label{eq:kappatheta} k(s + \epsilon)= k(s) \theta(\epsilon).\end{equation}

{\red By Theorem \ref{theorem:perturbation}, the function $\theta$ is holomorphic in a  neighbourhood of $0$, hence $\mathcal{C}^\infty$ in $0$ and so is $k$, with $ k^{(n)}(s) = k(s) \theta^{(n)}(0)$}. Recalling that $\lambda(t)=e^{-itq}\theta(it)$, we obtain
$$ \lambda'(0)=i \frac{k'(s)}{k(s)} - iq, \qquad \lambda^{(2)}(0) = -q^2 + 2 q \theta'(0) - \theta^{(2)}(0) = - q^2 + 2q \frac{k'(s)}{k(s)} - \frac{k^{(2)}(s)}{k(s)}.$$
Since $\lambda'(0)=0$, the assertions follow.
\end{proof}

\section{The Edgeworth expansion}\label{sect:edgeworth}
In this section we are going to prove a third-order Edgeworth expansion for $S_n$ w.r.t. the measure $\Qxs$, valid for all types of matrices.
We fix real $s \in \interior{I_\mu}$, and denote by $q:= \E_{\QQ^s} S_1$ the stationary drift of $(S_n)_{n \in \N}$. We will use the operator
$R(t) f(x) = \E_\Qxs [ e^{it (S_1 - q)} f(X_1)].$



Let
$$ F_{n,x}(t) := \Qxs \left\{ \frac{S_n-nq}{\sqrt{n \sigma^2}} \le t \right\}.$$
be the cumulative distribution function {\red of the standardized version of $S_n$, and write $\Phi$ for the cumulative distribution function of the standard normal distribution}. Then we have the following result.

\begin{thm}\label{thm:edgeworth}
Assume that $\mu$ satisfies $\condC$ and is non-arithmetic, or that \ipo~ or \ide~ holds. Assume moreover that \eqref{eq:iotamoment} holds for some $\epsilon >0$. Then
$$ \lim_{n \to \infty} \, \sup_{x \in \Sp} \, \left\{ \sqrt{n} \ \sup_{u \in \R} \abs{ F_{n,x}(u)  - \Phi(u) - \frac{m_3}{6 \sigma^3 \sqrt{n}}(1-u^2) \phi(u) + \frac{b(x)}{\sigma \sqrt{n}} \phi(u) } \right\} =0,$$
for quantities $b(x) \in \R$, $\sigma^2 > 0$, $m_3 \in \R$ as defined in \eqref{eq: sigma} and \eqref{eq: bx}.
\end{thm}

%
%
%
%

\begin{proof}[Proof of Theorem \ref{thm:edgeworth}] We proceed as in \cite{Herve2010}, i.e. we try to follow the standard proof
as in the i.i.d. case. Recall that, since we assume non-arithmeticity, $\sigma^2>0$ due to Corollary \ref{cor:arithmetic}.

{\sc Step 1}. We define the function
$$
{\red G_n(u) ~:=~ \Phi(u) + \frac{m_3}{6\sigma^3 \sqrt n} (1-u^2)\phi(u) - \frac{b(x)}{\sigma \sqrt n}\phi(u) ~=~ \Phi(u) - \frac{m_3}{6\sigma^3 \sqrt n} \phi''(u) - \frac{b(x)}{\sigma \sqrt n}\phi(u),\qquad u\in \R.}
$$ Here $\Phi$ denotes the {\red cumulative distribution function}, and $\phi$ the densitiy function of a standard normal distribution.
One can easily see that the derivative of $G_n$,
$$ {\red G_n'(u) ~=~ \phi(u) - \frac{m_3}{6 \sigma^3 \sqrt{n}} \phi^{(3)}(u) - \frac{b(x)}{\sqrt{n}} \phi'(u) }$$
  has exponential decay both at $+\8$ and $-\8$, uniformly in $n$. 
  Let  {\red $\gamma_n(t) := \int e^{itu}\, G_n'(u)\, du$ }  be the Fourier transform of  $G_n'$, then
  $$
  \gamma_n(t) = \bigg( 1+\frac{m_3}{6 \s^3 \sqrt n}(it)^3 \bigg) \cdot e^{-\frac 12 t^2} + \bigg( it \frac{b(x)}{\s \sqrt n} \bigg) e^{-\frac 12 t^2}
  $$
Denote
\begin{eqnarray*}
  \gamma_{0,n}(t) &:=& \bigg( 1+\frac{m_3}{6\s^3 \sqrt n}(it)^3 \bigg) \cdot e^{-\frac 12 t^2}, \\
  \gamma_{x,n}(t) &:=&  \bigg( it \frac{b(x)}{\s\sqrt n} \bigg) e^{-\frac 12 t^2}, \\
  \varphi_{n,x}(t) &:=& (R(t))^n(\1)(x) = \E_{\Qxs}[ e^{it (S_n-nq)}], \\
   {\red m } &:=&{\red \sup_{n \in \N} \sup_{u \in \R} \abs{G_n'(u)} < \infty.}
\end{eqnarray*}
By the Berry-Essen inequality (see \cite[XVI.(3.13)]{Feller1971}) 
we have that for all $T >0$,
\begin{equation}
\label{eq: BE}
\sup_{u\in \R} \big| F_{n,x}(u) - G_n(u) \big| \le \frac 1{\pi} \int_{-T}^T \bigg|\frac{\varphi_{n,x}(\frac t{\s\sqrt n}) -\gamma_n(t) }{t}\bigg|dt + \frac{24 m}{\pi T}.
\end{equation}
Next, fix $\eps >0 $, choose $a$ such that $\frac{24 m}{\pi a}<\eps $. Then with $T=a \sqrt{n}$,  $\frac{24 m }{\pi T}\le \frac{\eps}{\sqrt n}$.
We choose $\delta< \min\{a,\delta_1\}$, where $\delta_1$ is given by Corollary \ref{cor:Rt}, i.e. for $t \in (-\delta, \delta)$, the perturbation theory for $R(t)$ holds.

 Now we want to estimate the integral in $\eqref{eq: BE}$ by $O(\frac{\eps}{\sqrt n})$. For this purpose we divide the integral into two parts
\begin{eqnarray*}
A_n &=&  \int_{\s \delta \sqrt n\le |t| \le \s a \sqrt n}\bigg|\frac{\varphi_{n,x}(\frac t{\s \sqrt n}) -\gamma_n(t) }{t}\bigg|dt,\\
B_n &=&  \int_{|t|\le \s\delta \sqrt n} \bigg|\frac{\varphi_{n,x}(\frac t{\s \sqrt n}) -\gamma_n(t) }{t}\bigg|dt.
\end{eqnarray*}
{\sc Step 2.} We prove that $A_n\le \frac{\eps}{\sqrt n}$ for appropriately large $n$. By Lemma \ref{missing estimate}, we have for $u$ such that $\delta<|u|<a$ and all $x$ the estimate $|\varphi_{n,x}(u)| = |(R(u))^n(\1)(x) | \le \rho^n$,  hence
$$
\int_{\s\delta \sqrt n \le |t| \le \s a\sqrt n}\frac{\varphi_{n,x}(\frac t{\s \sqrt n})}{|t|}dt =
\int_{\delta  \le |u| \le a}\frac{\varphi_{n,x}(u)}{|u|}du \le C(a,\delta) \rho ^n.
$$ Moreover
$$
\int_{\s \delta \sqrt n \le |t| \le \s a\sqrt n} \frac{|\gamma_n(t)|}{|t|}dt \le C e^{-\sqrt n}.
$$
{\sc Step 3. } Now we estimate the last term $B_n$ to be smaller than $\frac{\eps}{\sqrt n}$. By Corollary \ref{cor:Rt} we write for $\frac{|t|}{\s \sqrt n}<\delta$ (recall that for such small values, the perturbation theory applies) \begin{eqnarray*}
\varphi_{n,x}\Big( \frac t{\s \sqrt n}\Big) - \gamma_n(t) &=& \lambda^n\Big(  \frac t{\s \sqrt n} \Big) \Pi\Big(  \frac t{\s \sqrt n} \Big) \1(x)
+N^n \Big(  \frac t{\s \sqrt n} \Big) \1(x) - \gamma_{0,n}(t) - \gamma_{x,n}(t)\\
&=& \bigg( \lambda^n\Big(  \frac t{\s \sqrt n} \Big)  -\gamma_{0,n}(t)  \bigg) + \lambda^n\Big(  \frac t{\s\sqrt n} \Big)\bigg(
\Pi\Big(  \frac t{\s \sqrt n} \Big) \1(x) - 1 - i t \frac{b(x)}{\s \sqrt n}
\bigg)\\ &&+ it \frac{b(x)}{\s \sqrt n}\bigg(  \lambda^n\Big(  \frac t{\s \sqrt n} \Big) -e^{-\frac 12 t^2}  \bigg)
+ N^n\Big( \frac t{\s \sqrt n}\Big)\1(x)\\
&=& I_1(t) + I_2(t,x) + I_3(t,x)+ I_4(t,x).
\end{eqnarray*}
Thus, we have to estimate four expressions. {\red For this purpose we will use the Taylor expansion
$$\lambda(u) ~=~ 1 - \frac{\sigma^2}{2} u^2 - i \frac{m^3}{6}u^3 + o(u^3),$$ given in Lemma \ref{lem: lambda}. The function $f(u)=\log \lambda(u) + \frac{\sigma^2}{2}u^2$ then satisfies $f(0)= f'(0)=f''(0)=0$ and $f^{(3)}(0)=-im_3$ and hence
$$  n f\big(\frac{t}{\sigma\sqrt{n}} \big) ~=~ -i \frac{m_3 t^3}{6 \sigma^3 \sqrt{n}} + o(t^3/\sqrt{n}). $$
Moreover, by choosing $\delta$ small enough (but fixed!), we can achieve that for all $u \in (-\delta, \delta)$,
$$ \abs{f(u)} \le \frac{1}{4} u^2 \quad \text{ and } \quad \abs{\frac{m_3}{6}u^3} \le \frac{1}{4} u^2, \quad \text{hence} \quad \max\left\{n \big|f (\frac{t}{\sigma\sqrt{n}}) \big|, \frac{m_3 t^3}{6 \sigma^3 \sqrt{n}}  \right\} ~\le~\frac{1}{4}t^2.$$ 
In particular, with this choice of $\delta$, $\abs{\lambda^n(t/(\sigma\sqrt{n}))} \le e^{-\frac{1}{4}t^2}.$

Considering now $I_1(t)$, we obtain, using the inequality 
\begin{equation}\label{eq:feller}\abs{e^u -1 - v} \le (\abs{u-v} + \frac{1}{2} \abs{v}^2 ) e^{\max(\abs{u}, \abs{v})},
\end{equation} which is valid for all $u, v \in \C$ (see \cite[XVI.(2.8)]{Feller1971}),
$$ \abs{I_1(t)} = e^{-\frac{1}{2} t^2} \abs{\exp \left( n f\big(\frac{t}{\sigma\sqrt{n}} \big) \right) - 1 + i \frac{m_3 t^3}{6\sigma^3 \sqrt{n}  }} ~\le~ e^{-\frac{1}{2}t^2} \left( t^3 o(\frac{1}{\sqrt{n}}) + t^6 O(\frac{1}{n}) \right) e^{\frac{1}{4} t^2}, $$
from which we infer that
$$ \int_{\abs{t} \le \sigma \delta \sqrt{n}} \, \abs{\frac{I(t)}{t}} dt ~\le~ \left(2 \int_0^\infty (t^3 + t^6) \, e^{-\frac{1}{4} t^2} dt \right) \cdot o\big( 1/ \sqrt{n} \big) $$

To estimate the integral of $I_2(t)$ we use
the bound on $\lambda$ from above, a second order Taylor expansion for $\Pi(t)$ (cf. Corollary \ref{cor:Rt}) and that $\Pi'(0)\1(x) = ib(x)$ (see \eqref{eq:defbx}). Then
$$ \abs{I_2(t,x)} \le e^{-\frac{1}{4} t^2} \abs{\Pi(0)\1(x) + \frac{t}{\sigma \sqrt{n}}\Pi'(0)\1(x) + O(\frac{t^2}{ n}) - 1 - it \frac{b(x)}{\sigma \sqrt{n}}} $$
and consequently
$$ \int_{|t|\le \sigma \delta \sqrt n} I_2(t,x) dt \le \int _{|t|< \sigma \delta \sqrt n} e^{-\frac{t^2}4} O\Big(\frac {t^2} n \Big)dt \le \frac Cn. $$

Turning to $I_3(t,x)$, we recall from Lemma \ref{lem: lambda}, that $b \in \B_\epsilon$, hence as a continous function on $\S$, it is bounded. Then
\begin{align*}
\abs{I_3(t,x)} ~\le~ \frac{t \normb{b}}{\sigma \sqrt{n}} e^{-\frac{1}{2}t^2} \abs{\exp\big( n f ( \frac{t}{\sigma \sqrt{n}}) \big) -1 } ~\le~ \frac{t \normb{b}}{\sigma \sqrt{n}} e^{-\frac{1}{4}t^2} \left( \frac{m_3 t^3}{6 \sigma^3 \sqrt{n}} + o(t^3/\sqrt{n}) \right),
\end{align*}
where we used again inequality \eqref{eq:feller}. Hence $\int \abs{I_3(t,x)}/t \, dt = O(1/n)$.

The integral over $I_4(t,x)$ is bounded independently of $x$ and vanishes at an exponential rate in $n$ since  $\|N(t)\|\le c(1-\eta)^n$ by Corollary \ref{cor:Rt}.
} 
%
\end{proof}

\section{The Bahadur-Rao Theorem for Products of Random Matrices}\label{sect:BahadurRao}

Now we are ready to prove our main result simultaneously for all types of matrices. We extend the approach for the one-dimensional case in \cite[Theorem 3.7.4]{Dembo1998}. Recall the definition $\Lambda(s)=\log k(s)$, such that $ \Lambda'(s) =\E_{\Q^s} S_1=: q$ and the Fenchel-Legendre transform of $\Lambda$ is given by$ \Lambda^*(q) = sq - \Lambda(s)$.

\begin{thm}
\label{thm:BahadurRao2}
Assume that $\mu$ satisfies $\condC$ and is non-arithmetic; or that \ipo~or \ide~hold.
Let $q=\E_{\QQ^s} S_1 = \frac{k'(s)}{k(s)} $ for some $s \in \interior{I_\mu}$ and assume there is $0 < \epsilon < 1$ such that
\eqref{eq:iotamoment} holds.
\begin{enumerate}
\item Then
\begin{equation}
\label{eq:br upper bound}
\limsup_{n \to \infty} \, \sup_{x \in \S} \, \sup_{d \in [0, \infty)} \, e^{sd} \frac{\sqrt{n}e^{snq}}{k(s)^n} {\red \Q_x}({S_n \ge nq+d}) < \infty.
\end{equation}
\item Consequently, there is $C < \infty$ s.t.~ for all $n \in \N$ and thereupon for each $u \ge nq$
\begin{equation}
\label{eq:br1}
{\red\QQP[x]{S_n > u} }\le \frac {C k(s)^n}{\sqrt n e^{s u}}.
\end{equation}
\item For each fixed $\theta \ge 0$ it holds that
\begin{equation}
\label{eq:br uniform convergence}
\lim_{n \to \infty} \, \sup_{x \in \S} \sup_{d \in [0, \theta \sqrt{n})} \, \abs{ s \sigma \sqrt{2 \pi n}  \, \frac{e^{s(nq +d)}}{k(s)^n} \, e^{\frac{d^2}{2 \sigma^2 n}} \E_{\Q_x}\Big[ \es(X_n) \1[{\{ S_n \ge nq+d\}}] \Big]- \es(x)}  = 0.
\end{equation}
\item In particular,  for all $x \in \S$,
\begin{equation}
\lim_{n \to \infty} \, s \sigma \sqrt{2 \pi n} \,e^{n \Lambda^*(q)} \E_{\Q_x}\Big[ \es(X_n) \1[{\{ S_n \ge nq\}}] \Big] = \es(x).
\end{equation}
\end{enumerate}
\end{thm}

Note that, using just the Chebyshev inequality and the definition of $k(s)$, one obtains in \eqref{eq:br1} the weaker upper bound
$$ {\red \QQP[x]{S_n > u}} \le \frac{C k(s)^n}{e^{su}},$$
where the factor $1/\sqrt{n}$ does not appear.

\begin{proof}
All the results will be consequences of a general argument. Fix $\theta \ge 0$, but let $d \ge 0$ be arbitrary for the time being. Introduce $\Psi_n := s \sigma \sqrt{n}$ and
$$ J_n^d ~:=~ s \sigma \sqrt{2 \pi n} \frac{e^{snq} e^{sd}}{k(s)^n}~=~ \Psi_n \, \sqrt{2 \pi} \, e^{n \Lambda^*(q)} \, e^{sd}$$
as well as the normalized quantity
$$ W_n ~:=~ \frac{S_n - nq}{\sqrt{n}\sigma},   $$
then $\Prob_{\Qxs}(W_n \le t ) = F_{n,x}(t)$. We obtain that
\begin{align*}
\frac{1}{\es(x)} \E_{\Q_x} \left({\es(X_n) \, \1[{\{S_n \ge nq + d \}}]}\right) ~&=~ \Erw[{\Qxs}]{e^{n\Lambda(s)-sS_n} \, \1[{\{S_n \ge nq + d \}}]} \\
&=~ e^{- n \Lambda^*(q)} \Erw[{\Qxs}]{e^{-s(S_n - nq)} \, \1[{\{S_n - nq\ge d \}}]}   \\
&=~ e^{- n \Lambda^*(q)} \Erw[{\Qxs}]{e^{-\Psi_n W_n} \, \1[{\{W_n \ge \frac{d}{\sqrt{n} \sigma} \}}]}
\end{align*}
Using the definition of $J_n^d$, we obtain
\begin{align*}
&~ J_n^d \frac{1}{\es(x)} \E_{\Q_x} \left({\es(X_n) \, \1[{\{S_n \ge nq + d \}}]}\right) \\
=&~ \sqrt{2 \pi} \Psi_n \, e^{sd}\, \int_{\frac{d}{\sigma \sqrt{n}}}^{\infty} e^{-\Psi_n t} \, dF_{n,x}(t) \\
=&~ \left. \sqrt{2 \pi} e^{sd} \Psi_n  e^{-\Psi_n t} \, F_{n,x}(t) \right|_{\frac{sd}{\Psi_n}}^{\infty} + \sqrt{2 \pi}  \, e^{sd}\, \int_{\frac{sd}{\Psi_n}}^{\infty} \Psi_n^2 e^{-\Psi_n t} \, F_{n,x}(t) \, dt\\
=&~  - \sqrt{2 \pi} \Psi_n \, F_{n,x}\left(\frac{sd}{\Psi_n}\right)  + \sqrt{2 \pi}  \, e^{sd}\, \int_{sd}^{\infty} \Psi_n e^{- t} \, F_{n,x}\left(\frac{t}{\Psi_n} \right) \, dt\\
=&~   \sqrt{2 \pi}  \, e^{sd}\, \int_{sd}^{\infty} \Psi_n e^{- t} \, \left[ F_{n,x}\left(\frac{t}{\Psi_n} \right) - F_{n,x}\left(\frac{sd}{\Psi_n}\right)\right]\, dt\\
\end{align*}
Defining  $h(t) := (1-t^2)\phi(t)$ and setting as before
$$ {\red G_n(t) ~:=~ \Phi(t) + \frac{m_3}{\sigma^3 \sqrt{n}}(1-t^2)\phi(t) - \frac{b(x)}{\sigma \sqrt{n}} \phi(t) ~=~  \Phi(t) + \frac{m_3}{\sigma^3 \sqrt{n}}h(t) - \frac{b(x)}{\sigma \sqrt{n}} \phi(t) },$$
we want to use the Edgeworth expansion from Theorem \ref{thm:edgeworth} in order to calculate the asymptotics. Therefore,
\begin{align*}
&~ J_n^d \frac{1}{\es(x)} \E_{\Q_x} \left({\es(X_n) \, \1[{\{S_n \ge nq + d \}}]}\right) \\
=&~ \sqrt{2 \pi}  \, e^{sd}\, \int_{sd}^{\infty}  e^{- t} \, s \sigma \sqrt{n} \, \left( \left[ F_{n,x}\left(\frac{t}{\Psi_n} \right) - G\left(\frac{t}{\Psi_n} \right) \right] - \left[ F_{n,x}\left(\frac{sd}{\Psi_n}\right) - G\left(\frac{sd}{\Psi_n}\right) \right] \right)\, dt &~(=: I_1^d(n,x))\\
&~+  \sqrt{2 \pi}  \, e^{sd}\, \int_{sd}^{\infty} \Psi_n e^{- t} \, \left[ \Phi\left(\frac{t}{\Psi_n} \right) - \Phi\left(\frac{sd}{\Psi_n}\right)\right]\, dt &~(=: I_2^d(n)) \\
&~+  \frac{m_3\sqrt{2 \pi}}{\sigma^3 \sqrt{n}}  \, e^{sd}\, \int_{sd}^{\infty} \Psi_n e^{- t} \, \left[ h\left(\frac{t}{\Psi_n} \right) - h\left(\frac{sd}{\Psi_n}\right)\right]\, dt &~(=: I_3^d(n)) \\
&~-  \frac{b(x)\sqrt{2 \pi}}{\sigma \sqrt{n}}  \, e^{sd}\, \int_{sd}^{\infty} \Psi_n e^{- t} \, \left[ \phi\left(\frac{t}{\Psi_n} \right) - \phi\left(\frac{sd}{\Psi_n}\right)\right]\, dt &~(=: I_4^d(n)) \\
\end{align*}

\medskip

It follows from Theorem \ref{thm:edgeworth} that $$ \lim_{n \to \infty} \sup_{x \in \S} \, \sup_{d \ge 0} \,  \abs{I_1^d(n,x)} = 0.$$
Using mainly that $\phi$ and $h$ have bounded derivatives, we are going to show that as well
\begin{equation}\label{eq:estimateI3} \lim_{n \to \infty} \sup_{d \ge 0} \abs{I_3^d(n)} = \lim_{n \to \infty} \sup_{d \ge 0} \abs{I_4^d(n)}  =0. \end{equation}
Finally, considering $I_2^d(n)$, we are going to obtain two different estimates, namely
\begin{equation} \label{eq:estimateI2}\abs{I_2^d(n)} \le  e^{-\frac{d^2}{2 \sigma^2 n}} \le 1, \end{equation}
and the refined estimate
\begin{equation} \lim_{n \to \infty} \sup_{d \in [0, \theta \sqrt{n}]} \abs{e^{\frac{d^2}{2 \sigma^2 n}} I_2^d(n) - 1} = 0. \label{eq:goodestimateI2}\end{equation}
Using the estimate \eqref{eq:estimateI2} allows to infer the upper bound \eqref{eq:br upper bound}, while the convergence result \eqref{eq:br uniform convergence} follows by using estimate \ref{eq:goodestimateI2}.
So it remains to prove Eqs. \eqref{eq:estimateI3} -- \eqref{eq:goodestimateI2}.

\medskip

\Step[2]: We consider $I_3^d$ and omit $I_4^d$, which can be treated along similar lines. A simple calculation shows that $h$ has a continuous derivative $h'$ with $\sup_{x \in \R} \abs{h'(x)} =: M < \infty$.
We compute
\begin{align*}
\abs{I_3^d(n)} ~=&~ \abs{ \frac{m_3\sqrt{2 \pi}}{\sigma^3 \sqrt{n}} \, e^{sd} \, \int_{sd}^{\infty} e^{-t} \, \left( \int_{sd/\Psi_n}^{t/\Psi_n} h'(r) \, \Psi_n  \, dr \right) \, dt} \\
~=&~ \abs{ \frac{m_3\sqrt{2 \pi}}{\sigma^3 \sqrt{n}} \, e^{sd} \, \int_{sd}^{\infty} e^{-t} \, \left( \int_{sd}^{t} h'\left( \frac{r}{\Psi_n}\right)   \, dr \right) \, dt}
=~ \abs{ \frac{m_3\sqrt{2 \pi}}{\sigma^3 \sqrt{n}} \, e^{sd} \, \int_{sd}^{\infty} \, h'\left( \frac{r}{\Psi_n}\right) \,  \int_{r}^{\infty} e^{-t}   \, dt  \, dr } \\
\le&~  \frac{m_3\sqrt{2 \pi}}{\sigma^3 \sqrt{n}} \, e^{sd} \, \int_{sd}^{\infty} \, \abs{h'\left( \frac{r}{\Psi_n}\right)} e^{-r}  \, dr
~\le~  \frac{m_3\sqrt{2 \pi}}{\sigma^3 \sqrt{n}} \, e^{sd} \, \int_{sd}^{\infty} \, M e^{-r}  \, dr  =  M\,  \frac{m_3\sqrt{2 \pi}}{\sigma^3 \sqrt{n}}
\end{align*}

\Step[3]: We are going to prove \eqref{eq:estimateI2} and \eqref{eq:goodestimateI2}. Therefore,
\begin{align}
I_2^d(n) ~=&~ e^{sd} \, \int_{sd}^\infty \Psi_n e^{-t} \, \left( \int_{sd/\Psi_n}^{t/\Psi_n} \, e^{-r^2/2} \, dr\right) \, dt ~=~ e^{sd} \, \int_{sd/\Psi_n}^\infty \Psi_n e^{-r^2/2} \, \left( \int_{\Psi_n r}^{\infty} \, e^{-t} \, dt\right) \, dr \\
=&~ e^{sd} \, \int_{sd/\Psi_n}^\infty \Psi_n e^{-r^2/2} \, e^{- \Psi_n r} \, dr ~=~ e^{sd} \,\left[ \left. - e^{- \Psi_n r  - r^2/2} \right|_{sd/\Psi_n}^\infty -  \int_{sd/\Psi_n}^\infty r e^{-r^2/2 - \Psi_n r} \, dr \right] \\
=&~ e^{-\frac{d^2}{2 \sigma^2 n}} -  \, \int_{sd/\Psi_n}^\infty r \, e^{sd - \Psi_n r} \, e^{-r^2/2 } \, dr \label{eq:I2est}
\end{align}
For all $d \ge 0$, we have
$$ 0 \le \int_{sd/\Psi_n}^\infty r \, e^{sd - \Psi_n r} \, e^{-r^2/2 } \, dr \le \int_{sd/\Psi_n}^\infty r \, e^{-r^2/2 } \, dr ~=~  e^{-\frac{d^2}{2 \sigma^2 n}}$$
and thus \eqref{eq:estimateI2} follows.

\Step[4]: In order to prove \eqref{eq:goodestimateI2}, let $\epsilon > 0$ be arbitrary and choose $\delta$ such that $\delta + \frac{\theta \delta}{\sigma^2} + \frac{\delta^2}{\sigma^2} < \epsilon$. We separate the last integral in Eq. \eqref{eq:I2est} into
$$ \int_{sd/\Psi_n}^{sd/\Psi_n + \delta/\sigma} r \, e^{sd - \Psi_n r} \, e^{-r^2/2 } \, dr   + \int_{sd/\Psi_n + \delta/\sigma}^\infty r \, e^{sd - \Psi_n r} \, e^{-r^2/2 } \, dr  ~=:~ A(n) + B(n)$$
and see that by the restriction $d \le \theta \sqrt{n}$, it holds that $sd / \Psi_n \le {\theta}/{\sigma}$ and thus
$$ A(n) \le \frac{\delta}{\sigma} \frac{\theta + \delta}{\sigma} e^{-\frac{d^2}{2 \sigma^2 n}}.$$
Finally,
$$ B(n) \le e^{-s \delta \sqrt{n}} \int_{sd/\Psi_n + \delta/\sigma}^\infty r \,  \, e^{-r^2/2 } \, dr = e^{-s \delta \sqrt{n}}   \, e^{-(sd/\Psi_n + \delta/\sigma)^2/2 } \le e^{-s \delta \sqrt{n}} \, e^{-\frac{d^2}{2 \sigma^2 n}} .$$
Upon choosing $n_0$ such that $e^{- s \delta \sqrt{n}} \le \delta$ for all $n \ge n_0$, we obtain that for all $n \ge n_0$, $0 \le A(n) + B(n) \le \epsilon\,  e^{-\frac{d^2}{2 \sigma^2 n}}$.
Thus we have proven that for all $\epsilon > 0$, there is $n_0$ such that for all $n \ge n_0$,
$$ \abs{e^{\frac{d^2}{2 \sigma^2 n}}I_2^d(n) - 1 } \le \epsilon.$$
\end{proof}

\begin{proof}[Proof of Theorem \ref{thm:BahadurRao}]
The main result is given above, while the formulas for $\sigma^2$ follow from Lemma \ref{lem: lambda} and Corollary \ref{cor:k}.
\end{proof}

\section{Tails of Stationary Solutions of Random Difference Equations}\label{sect:rde}

This section is devoted to Theorem \ref{thm:rde}. We start by giving an example for a matrix recursion from financial time series.

\begin{exa}
Consider the ARCH(2) process $Y_n$ defined by
$$ Y_n ~=~ \sigma_n \epsilon_n, \qquad \sigma_n^2 ~=~ a_1 Y_{n-1}^2 + a_2 Y_{n-2}^2 + 1, $$
where $\epsilon_n$ are i.i.d.~standard normal distributed random variables and $a_1, a_2 >0$ with $a_1 + a_2 < 1$.
\begin{enumerate}
\item Considering the squared process $(Y_n^2)_n$, we obtain a matrix recursion:
\begin{equation}
\left( \begin{array}{c}
\sigma_n^2 \\ Y_{n-1}^2
\end{array} \right) ~=~
\left( \begin{array}{cc}
a_1 \epsilon_{n-1}^2 & a_2 \\
\epsilon_{n-1}^2 & 0
\end{array} \right) \
\left( \begin{array}{c}
\sigma_{n-1}^2 \\ Y_{n-2}^2
\end{array} \right) +
\left( \begin{array}{c}
1 \\ 0
\end{array} \right) ~=:~ {\red \mM_n} \vec{Y_{n-1}} + B_n,
\end{equation}
{\red where $(\mM_n, B_n)_{n \in \N}$ is an i.i.d.~sequence in $\Mset \times \Rdnn$.}
{\red The matrix $\mA_1 := \mM_1^\top$ satisfies condition $\condC$ (it suffices to assume that allowable matrices have full measure), and the moment condition \eqref{eq:iotamoment} is readily checked, since 
\begin{align*}
\iota(\mA_1)^2 ~=~ \min_{x \in \Sp} \norm{\mA_1 x}^2 ~=~ \min_{\begin{subarray}{c}{x_1^2+x_2^2=1} \\ x_1, x_2 \ge 0 \end{subarray}} \, \left( a_1 \epsilon_0^2 x_1 + \epsilon_0^2 x_2 \right)^2 + a_2^2 x_1^2 ~\ge~ \big(\min\{a_1 \epsilon_0^2, a_2\} \big)^2.
\end{align*}
Consequently, 
$\iota(\mA_1) \ge \min\{a_1 \epsilon_0^2, a_2 \}$, which has all negative moments up to order 1/2 since $\epsilon_1$ is a standard normal random variable.}
Kesten \cite[Theorem 3]{Kesten1973} gives the following sufficient condition for the existence of $\alpha >0$ such that $k(\alpha)=1$: There is $s_0 >0$ such that
$$ \E{ \left[ \big( \min_{i} \ \sum_{j} {\red (\mM_1)_{i,j}} \big)^{s_0} \right]} ~\ge~ d^{s_0/2},$$
where $d$ is the dimension of the matrix {\red (this is not a misprint, Kesten's condition is stated in terms of the matrix $\mM_1$)}. In our case, we have the estimate
$$ {\red \E{ \left[ \big( \min_{i} \ \sum_{j} (\mM_1)_{i,j} \big)^s \right]} ~=~ \E \left[ \big( \min\{ a_1 \epsilon_0^2 + a_2, \epsilon_0^2\} \big)^s \right] ~\ge~ a_1^s \E (\epsilon_0^{2s}) }.  $$
Since $\epsilon_1$ is unbounded, the right hand side tends to infinity as $s$ grows, thus there is $\alpha >0$ with $k(\alpha) >0$. Finally, the non-arithmeticity assumption holds since $\epsilon$ has a continuous distribution, and eigenvalues depend continuously on the entries of a matrix.
\item Kl\"uppelberg and  Pergamenchtchikov showed in \cite[Lemma 2.7]{Klueppelberg2004}, that the process $(Y_n)_n$ has the same distribution as the process $(X_n)_n$ (if started with the same initial value), given by
$$ X_n ~=~ a_1 \eta_{1,n} X_{n-1} + a_2 \eta_{2,n} X_{n-2} + \eta_{3,n},$$
where $(\eta_{i,n})_n$ are independent sequences of i.i.d~standard normal random variables.
This leads to the matrix recursion
\begin{equation}
\left( \begin{array}{c}
X_n \\ X_{n-1}
\end{array} \right) ~=~
\left( \begin{array}{cc}
a_1 \eta_{1,n} & a_2 \eta_{2,n} \\
1 & 0
\end{array} \right) \
\left( \begin{array}{c}
X_{n-1} \\ X_{n-2}
\end{array} \right) +
\left( \begin{array}{c}
\eta_{3,n} \\ 0
\end{array} \right) ~=:~ {\red \mM_n} \vec{X_{n-1}} + B_n
\end{equation}
It can be shown that these matrices satisfy assumptions \ipo~as well as \ide, and it is proved in \cite[Lemma 3.2]{Klueppelberg2004} that there exists $\alpha > 0$ with $k(\alpha)=1$.
\end{enumerate}
\end{exa}

{\blue
Further instances of the equation $\R \eqdist \mM R + B$, with matrices satisfying the assumptions of Theorem \ref{thm:rde}, appear e.g.~ in \cite{Basrak2002} (GARCH-processes, nonnegative matrices), \cite{Wang2013} (multitype branching processes with immigration in random environment, nonnegative matrices), \cite{Behme2012} (stationary solutions of multivariate generalized Ornstein-Uhlenbeck processes, invertible matrices), to name just a few.
An extension of the methods used below applies to provide exact tail asymptotics for random variables $R$ which are fixed points of \emph{multivariate smoothing transforms}, i.e.
satisfying
\begin{equation}\label{eq:SFPE} R \eqdist \sum_{i=1}^N \mM_i R_i + B,\end{equation}
where $N \ge 2$ is a fixed integer, $R$ and $R_i$ are i.i.d.~and independent of the random matrices $\mA_i$ and the random vector $B$. The details are worked out in \cite{BM2015}.

Heavy tail properties of such $R$ were studied in \cite{BDG2011,BDMM2013,Mirek2013} and a result similar to \eqref{eq:RDE} was obtained, there $\alpha=\max\{s > 0 : \k(s) = 1/N\}$, but only in the first reference, which studies matrices satisfying $\condC$, it could be shown that $K >0$, in the latter two references, only partial results were obtained.}
%

{\red \subsection{Outlining the proof of Theorem \ref{thm:rde}}

Now we explain how the proof of Theorem \ref{thm:rde} is given by a sequence of lemmata, the proofs of which are {\red quite} technical and therefore postponed to the subsequent section, for a better stream of arguments. First, we have to introduce some notation.

{\red
\begin{notation}
Given a random element $(\mM, B) \in M(d \times d, \R) \times \R^d$, let $(\mM_n, B_n)_{n \in \N}$ be a sequence of i.i.d.~copies of $(\mM, B)$, defined on a probability space $(\Omega, \F, \Prob)$. Let $X_0 : \Omega \to \Sd$ be a random variable and $(\Prob_x)_{x \in \Sd}$ be a family of probability measures on $\Omega$, such that $(\mM_n, B_n)_{n \in \N}$ have the same law as under $\Prob$, while $\P[x]{X_0=x}=1$.
Write $\mG_n:= \mM_n^\top \cdots \mM^\top_1$, $X_n^*:=\mG_n \as X_0$, $S_n^* := \log \abs{\mG_n X_0}$.
\end{notation}

As mentioned before, we will apply the results obtained in the previous sections to the matrix $\mA_1 := \mM_1^\top$. Writing $\mu$ for the law of $\mA_1=\mM_1^\top$ under $\Prob$, and defining the measures $\Q_x$ as in Subsection \ref{sect:defQ}, we have the following identities, valid for all $x \in \S$:
\begin{align}
\Q_x \big( (X_0, (\mA_n)_{n \in \N}) \in \cdot \big) ~=~ \P[x]{ (X_0, (\mM_n^\top)_{n \in \N}) \in \cdot)} \\
\Q_x \big( (X_n, S_n)_{n \in \N} \in \cdot \big) ~=~ \P[x]{ (X_n^*, S_n^*)_{n \in \N}) \in \cdot)}.
\end{align}
If not explicitly stated otherwise, all appearing quantities below will be defined in terms of the sequences $(\mA_n)_{n \in \N}$, for example $k(s)$.
From now on, we fix $\alpha >0$ such that $k(\alpha)=1$ and set $q:=\E_{\Q^s}^\alpha S_1$. We will assume throughout that the assumptions of Theorem \ref{thm:rde} are in force.
}

Using the identifications from above, the SLLN in Proposition \ref{prop:FK} yields that
$$ \lim_{n \to \infty} \frac{1}{n} \log \norm{\mG_n} ~=~k'(0) <0 \qquad \Pfs, $$
which allows to infer that there is a unique solution (in distribution) to the equation $R \eqdist \mM R +B$, see e.g. \cite[Theorem 1.1]{Bougerol1992}.

\medskip

{\blue 
The fundamental idea  is to compare the behavior of $\skalar{x,R}$ with that of $\abs{\mG x} $. Therefore, we use that for $R_k$ being i.i.d.~copies of $R$ and independent of $(\mM_k, B_k)$, we have that for all $n \in \N$,
\begin{eqnarray*}
R&\eqdist& \mM_1 R_1+B_1 \eqdist \mM_1 \cdots \mM_n R_n + \sum_{k\le n} \mM_1 \cdots \mM_{k-1}B_k.
\end{eqnarray*}
Consequently, for any $x \in \Sd$,
\begin{eqnarray*}
\skalar{x,R} ~\eqdist~ \skalar{\red{x,\mM_1} \cdots \mM_n R_n} + \sum_{k\le n} \skalar{x,\mM_1 \cdots \mM_{k-1}B_k} ~\ge~ \skalar{\mG_n x, R_n} - \sum_{k\le n} \abs{\skalar{\mG_{k-1} x,B_k}}
\end{eqnarray*}
We are going to consider sets where first term dominates, while the sum is comparably small. In order to estimate the scalar product $\skalar{\mG_n x, R_n}$ from below by $\abs{\mG_n x}$, we will use the following lemma:
%

\begin{lem}\label{lem:cones}
Let the assumptions of Theorem \ref{thm:rde} hold.
 Then for all $D >0$ there are $J < \infty$, $\kappa_j>0$ and $c_j>0$,  $1 \le j \le J$, and disjoint subsets $\S_j \subset \Sd$, such that
\begin{equation}\label{eq:cones}\P{ \frac{R}{\abs{R}} \in \S_j \ \text{ and } |R|> \frac{D}{c_j}} ~\ge~ \kappa_j \end{equation}
and moreover
\begin{equation}
\label{eq: 4.2cones}
 \Rd \subset \bigcup_{j=1}^J \S_j^*,
 \end{equation}
where $\S_j^*$ are the cones
$$ \S_j^* := \{ y \in \Rd \, : \, \skalar{y,x} \ge c_j \abs{y}  \text{ for all $x \in \S_j$}\}. $$
If $\mu$ satisfies (C), then the same statement is valid, but for $\S_j$ being disjoint subsets of $\Sp$ and with \eqref{eq: 4.2cones} replaced by
$$ \R_{\ge} \subset \bigcup_{j=1}^J \S_j^*.
$$
\end{lem}
The proof of the lemma will be given in Section \ref{sect:rde proofs}.

The lemma now allows for the following comparison: If $R_n \in \S_j$ and $\mG_n x \in \S_j^*$, it follows that $\skalar{\mG_n x, R_n} \ge c_j \abs{\mG_n x} \abs{R_n}$. 

}}

\medskip

{\red As the next step, we use this comparison in more detail. }
Given constants $C_0, \delta >0$ (which will be chosen later), let $D = e^{C_0} \sum_{k=0}^\infty e^{-k \delta} = e^{C_0}/(1-e^{-\delta})$. Let $t \ge 0$ and define $n_t = \lceil \log t/q \rceil$.
\begin{eqnarray}
%
%
V_{n,t} &=& \Big\{ \red{S_n^*} \ge n_t q \ \mbox{ and } \ \log |B_{k+1}| + S_k^* \le n_t q + C_0  -(n-k)\delta  \ \forall k< n     \Big\} \label{v1}\\
V_{n,t}^j&=& V_{n,t} \cap \big\{ \mG_n X_0\in \S_j^*  \big\}\cap \big\{ R_n\in \S_j \ \mbox{ and } |R_n|> 2 \frac{D}{c_j}  \big\}, \label{v2}\\
\widetilde V_{n,t} &=& \bigcup_j  V_{n,t}^j. \label{v3}
\end{eqnarray}

Then we have the following lemma, the short proof of which we give immediately.

\begin{lem}
\label{lem:pos:12.3} For all $t \ge 0$,
$$
\P{\skalar{x,R}>Dt} \ge \P[x]{\bigcup_n \wt V_{n,t}}.
$$ Moreover, for all $n \in \N$,
$$\P[x]{\wt V_{n,t}}\ge (\min_j\, \kappa_j)  \,\P[x]{V_{n,t}} =: \kappa_0 \P[x]{V_{n,t}}.$$
\end{lem}

\begin{proof}
Recall that $\P[x]{X_0=x}=1$. Thus for every $n$ on the set $ V_{n,t}^j$ we have under $\Prob_x$
{\red \begin{eqnarray*}
\bigg\langle \mM_1 \cdots \mM_{n} R_n + \sum_{k\le n}\mM_1 \cdots \mM_{k-1}B_k,x \bigg\rangle &\ge&
\big\langle  R_n , \mG_n x \big\rangle - \sum_{k\le n}\big|\big\langle B_k, \mG_{k-1}x \big\rangle\big|\\
&\ge&
 c_j |R_n||\mG_n x|  - \sum_{k\le n}| B_k| |\mG_{k-1}x |\\
&\ge& 2D e^{n_t q} - e^{n_t q} e^{C_0} \sum_{k \le n}  e^{-(n-k)\delta}  \\
&\ge& Dt
\end{eqnarray*} }
To prove the second part of the Lemma we use the independence of  $\mG_n$ and $R_n$, the disjointness of $\S_j$ and the fact that $R \eqdist R_n$ to deduce
$$
{\red \P[x]{\wt V_{n,t}} = \pb_x \bigg(\bigcup_j  V^j_{n,t}\bigg) = \sum_j \P[x]{ V^j_{n,t}} \ge \kappa_0 \sum_j \P[x]{V_{n,t} \cap \big\{ \mG_n x\in \S_j^*  \big\} } = \kappa_0 \P[x]{V_{n,t}}.}
$$
\end{proof}

The final burden will then be to prove the following Lemma:

\begin{lem}\label{lem:finallemma}
There is $\eta >0$ and $T_0 >0$ such that for all $t \ge T_0$,
$${\red \P[x]{\bigcup_n \wt V_{n,t}} ~\ge~ \eta t^{-\alpha}.}$$
\end{lem}

In fact, in its proof, we will  for each fixed $t$ only consider a subset $K_t \subset \red{\N}$ of integers close to $n_t:= \lfloor \log t/q \rfloor$, with {\red $q = \E_{\Q^\alpha} S_1 > 0$}. We will choose $K_t \subset [n_t - \sqrt{n_t}, n_t]$ and prove that readily ${\red \P[x]{\bigcup_{k \in K_t} \wt V_{k,t}}} ~\ge~ \eta t^{-\alpha}$, using the inclusion--exclusion formula.

Therefore, we will use the following technical result, which finally fixes $C_0$ and $\delta$. It is here where the Bahadur-Rao theorem enters.

\begin{lem} \label{lem:estimate Vn}{\red Assume that $\alpha \in \interior{I_\mu}$ and that
\begin{equation}\label{eq:moment conditions2}
\E \norm{\mA}^{\alpha+\epsilon} \iota(\mA)^{-\epsilon} < \infty, \qquad 0 < \E \abs{B}^{\alpha + \epsilon} < \infty.
\end{equation}}
 Then
there are constants $\delta, C_0,D_1,D_2,N_0 >0$ such that {\red for all $x \in \S$}
$$
D_1 \cdot \frac{k(\alpha)^n}{\sqrt{n_t} e^{\alpha n_t q}}
\le {\red \P[x]{V_{n,t}}} \le
D_2 \cdot \frac{k(\alpha)^n}{\sqrt{n_t} e^{\alpha n_t q}}.
$$ for all $\lceil \log t / q \rceil = n_t > N_0$ and every $n_t-\sqrt{n_t} \le n \le n_t - \sqrt{n_t}/2$.

For the assertion of this lemma to hold, $k(\alpha) =1$ is not necessary, we only need that $k'(\alpha)>0$, then still $q:= \E_{\Q^\alpha} S_1$.
\end{lem}

Summing up what has been said before, we now are able to prove Theorem \ref{thm:rde}.

\begin{proof}[Proof of Theorem \ref{thm:rde}]
{\red We already mentioned that the assumptions of Theorem \ref{thm:rde} guarantee the existence and uniqueness (in distribution) of a solution $R$ to $R \eqdist \mM R+B$, see e.g.~\cite[Theorem 1.1]{Bougerol1992}.
The lower bound for the tail behavior then follows by combining Lemmas \ref{lem:pos:12.3} and \ref{lem:finallemma}.}
\end{proof}

\section{Proofs}\label{sect:rde proofs}

{\red
\subsection{On the support of $R$}
As mentioned before, in order to prove Lemma \ref{lem:cones}, we have to show that $\supp R$ is unbounded in ``enough'' directions. To make this statement precise, introduce
the {\em asymptotic support} of $R$:
 Consider the compactification $\overline{\Rd} := \Rd \cup \Sd_\infty$ of $\Rd$, with
$$ \Rd \ni x_n \text{ converges to } y \in \Sd_\infty = \Sd \quad \Leftrightarrow \quad \lim_{n \to \infty} \frac{x_n}{\abs{x_n}} = y \text{ and } \lim_{n \to \infty} \abs{x_n}=\infty $$
We will study the set
$$ V(R) := \{ y \in \Sd \, : \, \exists (r_n)_n \subset \supp \, R \, : \, \lim_{n \to \infty} \frac{r_n}{\abs{r_n}} = y \text{ and } \lim_{n \to \infty} \abs{\red{r_n}}=\infty \}.$$
Using diagonal sequences, one obtains that the set $V(R)$ is indeed closed and thus, as a subset of $\S$, even compact.
An important result is that the set $V(R)$ is invariant under the action of $\Gamma^*=[\supp \mM]$ on the sphere: Let $y \in V(R)$, with associated sequence $r_n$. Then $\mm r_n +b \in \supp \, R$ for all $(\mm,b) \in \supp(\mM,B)$, and still $\abs{\mm r_n + b} \to \infty$, with the ratio $\abs{\mm r_n +b}/\abs{\mm r_n}$ tending to one. Hence,
$$ \mm \as y = \lim_{n \to \infty} \left( \frac{\mm r_n}{\abs{ \mm r_n}} + \frac{b}{\abs{\mm r_n}} \right)= \lim_{n \to \infty} \frac{\mm r_n +b}{\abs{ \mm r_n + b}} , $$
and thus $\mm \as y \in V(R)$.

We have the following result about $V(R)$ for $\mA$ being nonnegative.
\begin{prop}\label{lem:supp nonnegative}
Under the assumptions of Theorem \ref{thm:rde}, let $\mA$ be nonnegative and satisfy $\condC$.
\begin{enumerate}
\item Assume that $B$ is nonnegative. Then $V(R) \cap \Sp \neq \emptyset$.
\item If $V(R) \cap \Sp \neq \emptyset$, then readily $V(R) \cap \interior{\Sp} \neq \emptyset$.
\end{enumerate}
\end{prop}

}


\begin{proof}
\Step[1]:
If $B$ is nonnegative as well, then $\supp R \subset \Rdnn$. Moreover, since $B \neq 0$, there is nonzero $r \in \supp R$.
But then as well ${\red R_n^r := } \mM_1 \cdots \mM_n r + \sum_{k \le n} \mM_1 \cdots \mM_{k-1} B_k \in \supp R$, in particular,
$$ \abs{R_n^r} ~\ge~ \abs{\mM_1 \cdots \mM_n r} ~\eqdist~ \abs{\mM_n \cdots \mM_1 r}.$$
We assumed that $\mA = \mM^\top$ satisfies $\condC$; but $\condC$ holds for $\mA$ if and only if it holds for $\mA^\top$. Hence Proposition \ref{prop:FK} applies and gives under $ ^*\Q_r^\alpha$ (which denotes the measure constructed in the same manner as $\Q_r^\alpha$, but using the law of $\mA^\top=\mM$),
$$ \lim_{n \to \infty} \frac{S_n}{n} ~=~  \lim_{n \to \infty} \frac{1}{n} \log \abs{\mM_n \dots \mM_1 r} ~=~ k'(\alpha)/k(\alpha) >0.$$
But finite marginal distributions of $^*\Q_r^\alpha$ and $\Prob$ are equivalent, and thus we have that the sequence $(\abs{\mM_n \dots \mM_1 r})_n$ is unbounded, hence $ \supp\, R$ is unbounded as well.

\Step[2]: 
As shown above, the set $V(R)$ is invariant under $\Gamma^*$. But by \cite[Lemma 4.3]{BDGM2014}, $V(\Gamma^*)$ is the unique closed minimal $\Gamma^*$-invariant subset, hence $V(R)$ contains $V(\Gamma^*)$. In particular, there is $y \in \interior{\Sp}$ with $y \in V(\Gamma^*) \subset V(R)$.
\end{proof}

{\red
\begin{proof}[Proof of Lemma \ref{lem:cones}]
We consider seperately the three cases of matrices.

\Case[1]: Assume $\mA$ is nonnegative and satisfies $\condC$ and that $\supp R \cap \Rdnn$ is unbounded. Then, by Lemma \ref{lem:supp nonnegative}, there is $y_0 \in \interior{\Sp} \cap V(R)$. It holds that $\min_{x \in \Sp} \skalar{x,y_0} \ge \min_{1 \le i \le d} (y_0)_i >0$. Let $\delta >0$ such that $B_\delta(y_0) \cap \Sp \subset \interior{\Sp}$. Then there is $c >0$ such that $\min_{x \in \Sp} \min_{y \in B_\delta(y_0)} \skalar{x,y} \ge c$, and we can set $J:=1$ and $\S_1 := B_\delta(y_0)$.
\medskip

\Case[2]: Assume that $\mA \in GL(d, \R)$, satisfying $\ide$. Due to the density assumption, which is invariant under taking the transpose, there is in particular a proximal matrix $\mm \in \Gamma^*$ with attracting eigenvector $v_\mm$. By \cite[Lemma 8.1]{AM2010}, the set $V(R)$ is non-empty, moreover, there are $y_1, y_2 \in V(R)$ such that
$\skalar{y_1, v_\mm} >0$ and $\skalar{y_2, (-v_\mm)}>0$. Since $V(R)$ is $\Gamma^*$-invariant, it follows that $\mm^n \as y_1$ and $\mm^n \as y_2$ are in $V(R)$ for all $n \in \N$, hence $v_\mm$ and $-v_\mm$ are in $V(R)$. 

Observe that the operator $\Pst[\alpha]$, defined in \eqref{eq:Pst}, leaves $\Cbf{V(R)}$ invariant. Moreover, due to the density assumption, the measure $\nu := (\Pst[\alpha])^{n_0} \frac{1}{2}( \delta_{v_\mm} + \delta_{-v_\mm})$ has a density with respect to the volume measure on $\Sd$, and in particular, gives mass zero to any hyperspace, and is still supported on (a subset of) $V(R)$ and is symmetric, i.e.~$\nu(A)=\nu(-A)$ for all $A \subset \Sd$. Then one can proceed as in \cite[Lemma 2.7 \& Lemma 2.8]{Guivarch2012}---these are the counterpart of Lemma \ref{lem:prop:qn} for invertible matrices---to show that $((\Pst[\alpha])^n) \nu (\1[V(R)]) \ge c \kappa(\alpha)^n =c$ for some $c>0$ and all $n \in \N$.
  
Together with the compactness of $V(R)$ this yields that, using Prokhorov's Theorem,
$$ \frac{1}{n} \sum_{k=0}^{n-1} (\Pst[\alpha])^k  \nu $$
is a weakly compact sequence and therefore has a subsequential limit $\nust[\alpha]$, which is a probability measure on $V(R)$ that satisfies $\Pst[\alpha] \nust[\alpha] = \nust[\alpha]$.

By Prop. \ref{prop:transferoperators}, \eqref{suppnus} and \eqref{nust}, it holds for all $x \in \S$ that
$$ \min_{x \in \S} \int_{V(R)} \, \abs{\skalar{x,y}}^s \, \nust(dy) ~=~ \min_{x \in \S} \frac{1}{c} \es(x) ~:=~\epsilon >0.$$ 
This shows that for all $x \in \Sd$ there is $y \in V(R)$ such that  $\skalar{x,y} \ge 2^{-1/s} \epsilon^{1/s}$ due to the symmetry of $\nust$. Hence we can find a partition of $V(R)$ into a finite number of sets (use compactness), such that the assertions of the Lemma hold.
\medskip

\Case[3]: Assume that $\mA \in GL(d,\R)$ satisfies $\ip$, and that there is no proper closed convex cone, which is $\Gamma^*$-invariant. Then it is shown \cite[after Theorem 5.1]{Guivarch2012}, that $V(R)$ contains the pre-image of $V(\Gamma^*)$ under the projection $\Sd \to \Pd$ (this is called Case I there). Using that $\mM=\mA^\top$ satisfies \ip~as well, if it is satisfied by $\mA$, we use \cite[Theorem 2.17]{Guivarch2012} to infer the existence of a symmetric probability measure $\nust[\alpha]$, which is supported in (a subset of) $V(R)$. Then we can conclude as above.
\end{proof} }

\subsection{Auxiliary Lemma}
Next we are going to prove Lemma \ref{lem:estimate Vn}.


\begin{proof}[Proof of Lemma \ref{lem:estimate Vn}]
\Step[1]: Denoting $U_{n,t} := \Big\{  S^*_n \ge n_t q  \Big\}$ and
\begin{eqnarray*}
W_{j,{n,t}} &:=& \Big\{ S^*_j +\log|B_{j+1}|> n_t q + C_0 - (n-j)\delta \Big\},\\
\end{eqnarray*}
we have that
$$
{\red \P[x]{V_{n,t}} = \P[x]{U_{n,t}} - \pb_x \bigg({\bigcup_{j<n} (U_{n,t}\cap W_{j,{n,t}})}\bigg)}.
$$
Using the Bahadur-Rao type result \eqref{eq:br uniform convergence}, we estimate ${\red\P[x]{U_{n,t}}}$ from below (with $d=q(n_t-n)$, $\theta = q+1$, $\vartheta_1 = \inf_{x,y} \frac{\est[\alpha](x)}{\est[\alpha](y)}$), namely, there is $N_0 \in \N$ and {\red$\vartheta_2 >0$}, such that for all $n \ge N_0$, the following estimate holds:
\begin{equation}
\label{eq: pos:star}
{\red \P[x]{U_{n,t}}} = \Q_x (S_n > n_t q) \ge \frac{\vartheta_1}{r_{\alpha}(x)} \E_{\Q_x} \Big[  \es[\alpha](X_n) {\bf 1}_{\{ S_n \ge nq + d \}} \Big] \ge \vartheta_2 \cdot
\frac{k(\alpha)^n}{\sqrt{n_t} e^{\alpha n_t q}}.
\end{equation}
 Similarly \eqref{eq:br1} provides an upper estimate for $\P{U_{n,t}}$ which in particular proves the upper bound in the lemma.

Therefore it is sufficient to prove that
\begin{equation}\label{aim}
{\red \pb_x} \bigg( \bigcup_{j<n} (U_{n,t}\cap W_{j,{n,t}}) \bigg) \le \vartheta \cdot \frac{k(\alpha)^n}{\sqrt{n_t} e^{\alpha n_t q}}
\end{equation}
for some $\vartheta<\vartheta_2/2$. In fact, we are going to show that {\red$\vartheta$} can be made arbitrarily small by choosing $C_0$ large.
\medskip

\Step[2]: Let us denote by $\mu_{\mA,B}$ the joint law of $(\mA,B)$. 
By decomposing the sets $W_{j,{n,t}}$ further, depending on the overshoot of $S_j^* + \log \abs{B_{j+1}}$, we obtain \begin{align*}
&~{\red \pb_x} \bigg(  \bigcup_{j<n} (U_{n,t}\cap W_{j,{n,t}}) \bigg)
 ~\le~ \sum_{j<n} {\red\P[x]{ U_{n,t}\cap W_{j,{n,t}}}}\\
=&~ \sum_{j<n}\sum_{m\ge 0} \pb \bigg( {
\frac{e^{n_t q + C_0  + m}}{e^{ (n-j)\delta}} \le |B_{j+1}| {\red |\mG_j x|} < \frac{e^{n_t q + C_0  + m+1}}{e^{ (n-j)\delta}}\ \mbox{ and }\ \red{ |\mG_n x|}> e^{n_tq}
}   \bigg) \\
\le &~ \sum_{j<n}\sum_{m\ge 0} \int  \pb\bigg(  {
\frac{e^{n_t q + C_0  + m}}{e^{ (n-j)\delta}} \le |b| |\mG_j x| < \frac{e^{n_t q + C_0  + m+1}}{e^{ (n-j)\delta}}\ \mbox{ and }\ \|\mG_{j+2}^n\|\|\ma^\top\||\mG_j x|> e^{n_tq}
} \bigg) \mu_{\mA,B}(d\ma,\, db) \\
\le&~  \sum_{j<n} \sum_{m\ge 0} \int
\pb \bigg(
 |\mG_j x| \ge \frac{ e^{n_t q + C_0 +m}}{ |b| e^{ (n-j)\delta}} \bigg)
 \cdot \pb \bigg(   \|\mPi_{n-j-1}\|>  \frac{|b|e^{ (n-j)\delta}}{\|\ma\| e^{C_0+  m + 1}}
  \bigg) \mu_{\mA,B}(d\ma,\,db)\\
\end{align*}

To estimate further, we have to consider separately the cases where $\abs{b}$ is small resp. large, for we can apply the Bahadur-Rao estimate only in the first case.
More precisely, we split the integral into two integrals over the set
\begin{equation} \label{eq:defTheta} \Theta := {\{|b| \le e^{(n_t-j)q + C_0 - (n-j)\delta + m}\}} \end{equation}
and its complement $\Theta^c$, respectively.

\medskip

\Step[3]: In this step, we estimate
$$ I := \sum_{j<n} \sum_{m\ge 0} \int \1[\Theta](b) \,
\pb \bigg(
 |\mG_j x| \ge \frac{ e^{n_t q + C_0 +m}}{ |b| e^{ (n-j)\delta}} \bigg)
 \cdot \pb \bigg(   \|\mPi_{n-j-1}\|>  \frac{|b|e^{ (n-j)\delta}}{\|\ma\| e^{C_0+  m + 1}}
  \bigg) \mu_{\mA,B}(d\ma,\,db).$$
{\red 
In this step, we also choose $\delta$. $C_0$ will be a free parameter until Step 5, and it is important to notice, that all appearing constants are independent of $C_0$.

  On $\Theta$, $e^u:= \frac{ \exp({n_t q + C_0 +m})}{ |b| \exp({ (n-j)\delta}) } \ge \exp(jq)$, thus  the estimate \eqref{eq:br1} applies to the first probability in $I$ and yields
\begin{equation}\label{eq:aa1} \P{ \abs{\mG_j x} \ge e^u } ~=~ \P[x]{S_n^* \ge u} ~=~ \QP[x]{S_n \ge u} ~\le~ \frac{C k(\alpha)^j}{\sqrt{j} \, e^{\alpha u}} ~=~ \frac{C k(\alpha)^j \abs{b}^\alpha e^{ \alpha (n-j) \delta}} {\sqrt{j}\, e^{\alpha (n_t q +C_0 +m)} }, \end{equation}
where $C$ is given by Theorem \ref{thm:BahadurRao2} and only depends on $\alpha$.

In order to estimate the second probability in $I$, we use the Markov inequality with the function $x \mapsto x^{\beta}$, where we choose $\beta >0$ such that $\beta < \alpha$ and $k(\beta) < k(\alpha)$. This is possible since $k'(\alpha)>0$. Then, by Corollary \ref{cor:ksEs}, 
$$ \E \norm{\mPi_{n-j-1}}^{\beta} ~\le~\frac{1}{c_\beta} k(\beta)^{n-j-1}  ~=~ \frac{1}{c_\beta k(\beta)} k(\alpha)^{n-j} \left( \frac{k(\beta)}{k(\alpha)} \right) ^{n-j}$$
with $c_\beta >0$ given by Prop. \ref{lem:prop:qn}.
We obtain, applying the Markov inequality as described above,
\begin{align}
\label{eq:aa2} \pb \bigg(   \|\mPi_{n-j-1}\|>  \frac{|b|e^{ (n-j)\delta}}{\|\ma\| e^{C_0+  m + 1}}
  \bigg)  ~ &\le ~ \frac{\E \big[\norm{\mPi_{n-j-1}}^\beta \big] \norm{\ma}^\beta e^{\beta(C_0 + m + 1)} }{\abs{b}^\beta e^{\beta(n-j)\delta}}\\ ~&\le~ \frac{k(\alpha)^{n-j} \norm{\ma}^\beta e^{\beta(C_0 + m + 1)} }{k(\beta) c_\beta \abs{b}^\beta e^{\beta(n-j)\delta}} \left( \frac{k(\beta)}{k(\alpha)} \right) ^{n-j}  \nonumber
\end{align}
Define $\xi:= \alpha - \beta>0$. Using the estimates \eqref{eq:aa1} and \eqref{eq:aa2} in $I$ and simplifying terms, we obtain
\begin{equation}
\label{eq:aa3} I ~ \le ~ \sum_{j<n} \sum_{m\ge 0} \int \1[\Theta](b) \, \frac{C e^\beta}{c_\beta k(\beta )} \frac{1}{\sqrt{j}} \frac{k(\alpha)^n \abs{b}^\xi \norm{\ma}^\beta e^{\xi(n-j)\delta}} {e^{\alpha n_t q} \,  e^{\xi(C_0 +m)} }  \,  \left( \frac{k(\beta)}{k(\alpha)} \right) ^{n-j}  \,  \mu_{\mA,B}(d\ma,\,db)
\end{equation}
We estimate further by omitting the indicator $\1[\Theta](b)$, integrating,  setting $C' := C e^\beta /c_\beta k(\beta)$ and using Fubini's theorem:
\begin{align}
\label{eq:aa4} I ~ \le& ~ \sum_{j<n} \sum_{m\ge 0}  C' \frac{1}{\sqrt{j}} \frac{k(\alpha)^n  e^{\xi(n-j)\delta}} {e^{\alpha n_t q} \,  e^{\xi(C_0 +m)} }  \,  \left( \frac{k(\beta)}{k(\alpha)} \right) ^{n-j}  \, \E \big( \norm{\mA}^\beta \abs{B}^\xi \big) \\
\nonumber =&~ \frac{C' k(\alpha)^n }{e^{\xi C_0} e^{\alpha n_t q}} \left( \sum_{m \ge 0} e^{-\xi m} \right)  \left( \sum_{j<n} \frac{1}{\sqrt{j}} \left( \frac{k(\beta) e^{\xi \delta}}{k(\alpha)} \right) ^{n-j} \right) \, \E \big( \norm{\mA}^\beta \abs{B}^\xi \big) 
  \end{align}
Recall that we chose $\beta <\alpha$ such that $k(\beta) < k(\alpha)$. Thus, we can choose a (small) $\delta >0$, such that 
\begin{equation} \label{eq:exi} e^{\xi \delta} k(\beta) < k(\alpha), \end{equation} and apply Lemma  \ref{lem:sumineq} (see below) with $\rho=e^{\xi \delta} k(\beta) /k(\alpha)$  to infer that the sum over $j$ is bounded by a constant times $1/\sqrt{n}$.
On $\E( \norm{\mA}^\beta \abs{B}^\xi)$ we can apply H\"older's inequality with $p_1=\alpha/(\alpha-\xi)=\alpha/\beta$ and $p_2= \alpha/ \xi$ to infer
$$ I ~\le~ \frac{C' k(\alpha)^n }{e^{\xi C_0} e^{\alpha n_t q}} \,  \frac{1}{1-e^{-\xi}} \, \frac{D_3}{\sqrt{n}} \, \E \norm{\mA}^\alpha \, \E \abs{B}^\alpha ~=:~ \frac{C'' k(\alpha)^n }{\sqrt{n} e^{\xi C_0} e^{\alpha n_t q}} ~\le~ \frac{C'''}{e^{\xi C_0}} \cdot \frac{ k(\alpha)^n }{\sqrt{n_t}  e^{\alpha n_t q}}  $$
for a finite constant $C'''$, which does not depend on $n$ or $t$ or $C_0$. Note that we were allowed to replace $n$ by $n_t$ in the final expression, since  $n_t-\sqrt{n_t} \le n \le n_t - \sqrt{n_t}/2$ by assumption.
}

\medskip
\Step[4]:
 Now, to estimate
 $$II~:=~\sum_{j<n} \sum_{m\ge 0} \int \1[\Theta^c](b) \,
\pb \bigg(
 |\mG_j x| \ge \frac{ e^{n_t q + C_0 +m}}{ |b| e^{ (n-j)\delta}} \bigg)
 \cdot \pb \bigg(   \|\mPi_{n-j-1}\|>  \frac{|b|e^{ (n-j)\delta}}{\|\ma\| e^{C_0+  m + 1}}
  \bigg) \mu_{\mA,B}(d\ma,\,db), $$ 
{\red
we start by applying the Markov inequality with $x\to x^\alpha$ resp. $x \to x^\beta$ with $\beta$ as above to both probabilities:
\begin{align}
\label{eq:b1} \pb \bigg(
 |\mG_j x| \ge \frac{ e^{n_t q + C_0 +m}}{ |b| e^{ (n-j)\delta}} \bigg) ~&\le~ \frac{\E \big[\abs{\mG_j x}^\alpha \big] \abs{b}^\alpha e^{\alpha(n-j)\delta}  }{e^{\alpha(n_t q + C_0 +m)}} ~\le~ \frac{k(\alpha)^j  \abs{b}^\alpha e^{\alpha(n-j)\delta}  }{c_\alpha e^{\alpha(n_t q + C_0 +m)}} \\
 \pb \bigg(   \|\mPi_{n-j-1}\|>  \frac{|b|e^{ (n-j)\delta}}{\|\ma\| e^{C_0+  m + 1}}
  \bigg) ~&\le~ \frac{\E \big[ \norm{\mPi_{n-j-1}}^\beta \big] \norm{\ma}^\beta \, e^{\beta(C_0 + m +1)}}{\abs{b}^\beta e^{\beta(n-j)\delta}} ~\le~ \frac{k(\beta)^{n-j}  \norm{\ma}^\beta \, e^{\beta(C_0 + m +1)}}{c_\beta k(\beta) \, \abs{b}^\beta e^{\beta(n-j)\delta}} ,
\end{align}
 were we used  as before Corollary \ref{cor:ksEs} to obtain the second inequalities. Hence, with $\xi=\alpha - \beta$ as before, \begin{align}
 II ~&\le~ \sum_{j<n} \sum_{m\ge 0} \frac{k(\alpha)^n e^\beta}{c_\alpha c_\beta k(\beta) e^{\xi{C_0}} e^{\alpha n_t q}} \, e^{-\xi m} \,   \left( \frac{e^{\xi \delta} k(\beta)}{k(\alpha)}\right)^{n-j}  \,  \int \1[\Theta^c](b) \, \norm{\ma}^\beta \abs{b}^\xi \,  \mu_{\mA,B}(d\ma,\,db) \\ \label{eq:bb3}
 &\le~ \frac{D k(\alpha)^n}{e^{\xi C_0} e^{\alpha n_t q}} \left( \sum_{m \ge 0} e^{-\xi m} \, \sum_{j <n} \left( \frac{e^{\xi \delta} k(\beta)}{k(\alpha)}\right)^{n-j} \, \int \1[\Theta^c](b) \, \norm{\ma}^\beta \abs{b}^\xi \,  \mu_{\mA,B}(d\ma,\,db) \right)
 \end{align}
 with $D= e^\beta /(c_\alpha c_\beta k(\beta))$.
Recall from \eqref{eq:defTheta} that the set $\Theta^c$ is defined in terms of $m$ and $j$, thus in order to resolve the sums, we first have to deal with the integral. 
We will apply the H\"older inequality twice. Choose (a small) $\gamma >1$ such that \begin{equation} \label{eq:gam}\E \norm{\mA}^{\gamma \alpha} < \infty \quad \text{ and } \quad\E \abs{B}^{\gamma \alpha} < \infty \quad \text{ and (still) } \quad e^{\xi \delta (1 + \frac{\gamma-1}{\gamma})} k(\beta) < k(\alpha).\end{equation} This is possible due to \eqref{eq:exi} and  the moment assumptions \eqref{eq:moment conditions2}. Then, using H\"older first with $p_1=\gamma$, $p_2=\gamma/(\gamma-1)$ and subsequently with $p_1'=\alpha/\beta$, $p_2'=\alpha/\xi$, we obtain
\begin{align}
\label{eq:bb2}\E \1[\Theta^c](B) \norm{\mA}^\beta \abs{B}^\xi ~&\le~  \left(\P{B \in \Theta^c} \right)^{\frac{\gamma-1}{\gamma}}  \, \left( \E \norm{\mA}^{\gamma \beta} \abs{B}^{\gamma \xi} \right)^{\frac{1}{\gamma}} \\ \nonumber
 ~&\le~  \left(\P{B \in \Theta^c} \right)^{\frac{\gamma-1}{\gamma}}  \, \left( \E \norm{\mA}^{\gamma \alpha} \right)^{\frac{\beta}{\alpha \gamma}}  \, \left( \E \abs{B}^{\gamma \alpha} \right)^{\frac{\xi}{\alpha \gamma}}
\end{align}
Now we apply the Markov inequality with $x \mapsto x^\xi$ to estimate $(\P{B \in \Theta^c})$,
\begin{align}
\label{eq:pbtc}  \P{\abs{B} > e^{(n_t-j)q + C_0 - (n-j)\delta + m} }  ~&\le~ \big( \E \abs{B}^\xi \big) e^{ - {\xi q}(n_t -j) - {\xi}(C_0 + m) } e^{ {\xi \delta} (n-j) } \\
\nonumber ~&\le~ \big( \E \abs{B}^\xi \big) e^{- \xi q \sqrt{n_t}} \, \cdot 1 \cdot e^{\xi \delta (n-j)},
\end{align}
 where the last inequality is valid for $j \le n$ (only such $j$ appear in the sum) and follows from the condition $ n < n_t - \sqrt{n_t}$.

Using \eqref{eq:pbtc} in \eqref{eq:bb2} and this in \eqref{eq:bb3}, we obtain
\begin{align}
II ~&\le~ \frac{D k(\alpha)^n}{e^{\xi C_0} e^{\alpha n_t q}} \left( \sum_{m \ge 0} e^{-\xi m} \, \sum_{j <n} \left( \frac{e^{\xi \delta} k(\beta)}{k(\alpha)}\right)^{n-j}  \frac{\big( \E \abs{B}^\epsilon \big)^{\frac{\gamma-1}{\gamma}}  \left( \E \norm{\mA}^{\gamma \alpha} \right)^{\frac{\beta}{\alpha \gamma}}  \, \left( \E \abs{B}^{\gamma \alpha} \right)^{\frac{\xi}{\alpha \gamma}}}{e^{\xi q\frac{ \gamma -1}{\gamma} \sqrt{n_t}}} e^{\xi \delta\frac{ (\gamma -1)}{\gamma} (n-j)} \right) \\
&=~ \frac{D' k(\alpha)^n}{e^{\xi C_0} e^{\alpha n_t q}} \frac{1}{e^{\xi q \frac{\gamma -1}{\gamma} \sqrt{n_t}}} \left( \sum_{m \ge 0} e^{-\xi m} \, \sum_{j <n} \left( \frac{e^{\xi \delta (1+ \frac{\gamma -1}{\gamma})} k(\beta)}{k(\alpha)}\right)^{n-j}    
\right)
\end{align}
for a finite constant $D'$, independent of $n, t, C_0$. Recalling \eqref{eq:gam}, both sums converge. Finally, up to a constant, the second quotient can be replaced by $1/\sqrt{n_t}$, and thus we arrive at
\begin{equation}\label{eq:estimateII} 
II ~\le~ \frac{D''}{e^{\xi C_o}} \frac{k(\alpha)^n}{\sqrt{n_t} e^{\alpha n_t q}}
\end{equation}

\medskip
\Step[5]: Recall that our original aim was to prove \eqref{aim} with an arbitrarily small $\vartheta$. From the previous steps, we have the estimate
$$\pb_x \bigg( \bigcup_{j<n} (U_{n,t}\cap W_{j,{n,t}}) \bigg) ~\le~ I + II ~\le~  \frac{C''' + D''}{e^{\xi C_0}} \frac{k(\alpha)^n}{\sqrt{n_t} \, e^{\alpha n_t q}},$$
where $C'''$ and $D''$ are finite constants, independent of $C_0$, which is still a free parameter. Thus, by choosing appropriately large $C_0$, we obtain the assertion. 
}
\end{proof}


{\red \begin{lem}\label{lem:sumineq}
Let $0 < \rho <1$, then there is $D_3 < \infty$ such that for all $n \in \N$,
$$ \sum_{j=0}^n  \frac{1}{\sqrt{j}} \, \rho^{n-j} ~\le~ D_3 \frac{1}{\sqrt{n}}$$
\end{lem}

\begin{proof}
As the first step, we relabel the sum to  $\sum_{j=0}^n  \frac{1}{\sqrt{n-j}} \rho^{j}$. Then we split the sum at $J_n:= 1/2 (\log n)/\log (\rho)$ and use that $\rho^{J_n}=1/\sqrt{n}$ and that $n /(n-J_n)$ converges to 1 as $n$ goes to infinity:
\begin{align*}
 \sum_{j=0}^{J_n}  \frac{1}{\sqrt{n-j}} \,  \rho^{j} + \sum_{j={J_n}}^n  \frac{1}{\sqrt{n-j}} \,  \rho^{j} ~&\le~ \frac{1}{\sqrt{n}} \left( \sum_{j=0}^{J_n} \sqrt{\frac{n}{n- J_n}} \, \rho^j \right) + \rho^{J_n} \left( \sum_{j={J_n}}^n \frac{1}{\sqrt{j}} \, \rho^{j-J_n} \right) \\
 &\le~ \frac{1}{\sqrt{n}} \left( \left[ \sup_{n \in \N} \sqrt{\frac{n}{n- J_n}} \right] \sum_{j=0}^{\infty}  \rho^j \right) + \frac{1}{\sqrt{n}} \left( \sum_{j={0}}^\infty  \rho^{j} \right) ~=: \frac{D_3}{\sqrt{n}}
\end{align*}
\end{proof}
}

\subsection{Finishing the proof}

\begin{proof}[Proof of Lemma \ref{lem:finallemma}]
{\red \Step[1]: We have to prove that
$
\pb_x ( \bigcup_{n \in \N} \wt V_{n,t} ) \ge \eta t^{-\alpha}
$ for some $\eta >0$ and all large $t$.  In order to do so, we can estimate the probability from below by $\P[x]{\bigcup_{n \in K_t} \wt V_{n,t}}$, where $K_t$ can be any subset of $\N$, and may depend on $t$.
Applying the inclusion-exclusion formula and Lemma \ref{lem:pos:12.3}, we obtain
\begin{align}
\nonumber \pb_x \bigg( \bigcup_{n\in K_t}\wt V_{n,t} \bigg) ~\ge &~ \sum_{n\in K_t} \pb_x \big({\wt V_{n,t}} \big) - \sum_{n,n'\in K_t:\; n>n'}\pb_x \big( \wt V_{n,t} \cap \wt V_{n',t} \big)\\ \label{pvd}
  ~\ge &~ \kappa_0 \sum_{n\in K_t} \pb_x \big({ V_{n,t}} \big) - \sum_{n,n'\in K_t:\; n>n'}\pb_x\big(  V_{n,t} \cap  V_{n',t} \big) ,
\end{align}
where we also used that $(\wt V_{n,t} \cap \wt V_{n',t}) \subset (V_{n,t} \cap V_{n',t})$, cf. their definitions in \eqref{v1} -- \eqref{v3}.
In order to make the second sum small, we will consider the following specific subsets $K_t$,
\begin{equation}
\label{eq:pos:4}
K_t= \big\{  kC_1:\; n_t-\sqrt{n_t} < kC_1 < n_t - \sqrt{n_t}/2  \big\},
\end{equation}
where the parameter $C_1$ will be chosen later.

\medskip
\Step[2]: In this step, we compute $\P{V_n,t \cap V_{n',t}}$ for $n,n' \in K_t$ with $n>n'$.
Let $\mG_{n'+1}^n = \mM_n^\top \dots \mM_{n'+1}^\top$.
\begin{align}
\P[x]{V_{n,t}\cap V'_{n,t}} ~\le&~ \P{ |\mG_{n'} x|\ge t \mbox{ and } |\mG_n x|\ge t} \nonumber\\
~\le&~ \sum_{m=0}^{\8} \P{te^m \le |\mG_{n'} x| < t e^{m+1} \mbox{ and } \|\mG_{n'+1}^n\||\mG_{n'} x| >t  }\nonumber \\
~\le&~ \sum_{m=0}^{\8} \P{|\mG_{n'} x| \ge  t e^{m}} \P  {\|\mPi_{n-n'}^\top\| > e^{-m-1}  } \label{vv1}
\end{align}
For the first probability, we can apply the estimate \eqref{eq:br1}: Rewriting it as $\QP[x]{S_{n'} > \log t + m}$, 
 we have since $n' \in K_t$ that $n' \le n_t= \lceil \log t/q \rceil $ and thus $u:=\log t +m \ge n' q$. Hence, 
$$ \P{|\mG_{n'} x| \ge  t e^{m}} ~\le~ \frac{C k(\alpha)^{n'}}{\sqrt{n'} \, t^\alpha e^{\alpha m}} ~\le ~ \frac{C' \cdot }{\sqrt{n_t} \,  t^\alpha e^{\alpha m}}.$$
For the second inequality, we used that $n' \in K_t$ is comparable to $n_t$.

For the second probability, we use the Markov inequality with $x \to x^\beta$, where $\beta < \alpha$ is such that $k(\beta) < k(\alpha)=1$ (as above), this is possible since $k'(\alpha)>0$. Together with Corollary \ref{cor:ksEs}, we obtain
$$ \P  {\|\mG_{n-n'}\| > e^{-m-1}  } ~\le~ e^{\beta(m+1)}{\E \norm{\mPi_{n-n'}}^\beta} ~\le~ e^\beta c_\beta^{-1} k(\beta)^{n-n'} \, e^{\beta m}.$$
Using these estimates in \eqref{vv1}, we infer
\begin{align*}
\P[x]{V_{n,t} \cap V_{n',t}} ~\le~  \frac{C' e^\beta}{c_\beta} \frac{  k(\beta)^{n-n'}}{\sqrt{n_t} t^\alpha } \sum_{m=0}^\infty e^{-(\alpha - \beta)m} ~\le~  \frac{C''  k(\beta)^{n-n'}}{\sqrt{n_t} \, t^\alpha}
\end{align*}

\medskip
\Step[3]: Using the previous step and the estimate from Lemma \ref{lem:estimate Vn} (again $k(\alpha)=1$) in \eqref{pvd}, we obtain 
\begin{align*}
 \pb_x \bigg( \bigcup_{n\in K_t}\wt V_{n,t} \bigg) ~\ge&~
   \kappa_0\sum_{n\in K_t} \frac{D_1}{\sqrt{n_t} e^{\alpha n_t q}} -
  \sum_{n\in K_t} \sum_{n'\in K_t:\; n'<n}  \frac{C''}{\sqrt{n_t}t^{\alpha}} \cdot k(\beta)^{n-n'}\\
 \ge&~ \frac{|K_t|}{\sqrt{n_t} t^{\alpha}}\big( \kappa_0 D_1 - C'' \sum_{n' \in K_t :\; n' < n} k(\beta)^{n-n'} \big) 
 \end{align*}
 Now we use that the cardinality $\abs{K_t} \ge \frac{\sqrt{n_t}}{2 C_1} -1 \ge \frac{\sqrt{n_t}}{4 C_1}$ for all sufficiently large $t$, and that $n - n' \ge C_1$ for $n, n' \in K_t$, $n >n'$. 
 \begin{align*}
 \pb_x \bigg( \bigcup_{n\in K_t}\wt V_{n,t} \bigg)
 \ge&~ \frac{1}{4 C_1}\left( \kappa_0 D_1 - C'' k(\beta)^{C_1} \,  \sum_{j=0}^\infty k(\beta)^{j} \right) \frac{1}{t^\alpha}  ~=~ \frac{1}{4 C_1}\left( \kappa_0 D_1 - \frac{C'' k(\beta)^{C_1}}{1-k(\beta)} \right) \frac{1}{t^\alpha}
 \end{align*}
 Since $k(\beta)<k(\alpha)=1$, we can now choose $C_1$ such that the term in the brackets becomes positive. Then, for all $t$ sufficiently large (in particular, such that $K_t$ is nonempty),
 $$\pb_x \bigg( \bigcup_{n\in K_t}\wt V_{n,t} \bigg) ~\ge~ \eta t^{-\alpha},$$
 with $\eta>0$ being independent of $t$.
}

\end{proof}

\end{document}